%% file: main.tex
\documentclass{article}

\usepackage{fullpage}
\usepackage{amsmath, amsthm, amssymb}
\usepackage{cleveref}
\usepackage{algorithm}
\usepackage{algorithmicx,algpseudocode}
\usepackage{mathtools}
\usepackage{thm-restate}
\usepackage{subfig}

\newtheorem{theorem}{Theorem}

\newtheorem{corollary}[theorem]{Corollary}
\newtheorem{definition}{Definition}

\newcommand{\Lap}{\text{Lap}}
\newcommand{\R}{\mathbb{R}}
\newcommand{\argmax}{\mathrm{argmax}}
\newcommand{\argmin}{\mathrm{argmin}}
\newcommand{\eps}{\epsilon}
\newcommand{\E}{\mathbb{E}}

\newcommand{\set}[1]{\left\{#1\right\}}
\newcommand{\abs}[1]{\left|#1\right|}
\newcommand{\ceil}[1]{\lceil#1\rceil}
\newcommand{\floor}[1]{\lfloor#1\rfloor}

\usepackage{xcolor}
\definecolor{DarkGreen}{rgb}{0.1,0.5,0.1}

\newcommand{\sk}[1]{\textcolor{orange}{[Sara: #1]}}

\newcommand{\ReportMax}{{\textsc{ReportMax}}}
\newcommand{\AboveThreshold}{{\textsc{AboveThreshold}}}
\newcommand{\PNCPD}{{\textsc{PNCPD}}}
\newcommand{\OnlinePNCPD}{{\textsc{OnlinePNCPD}}}

\begin{document}

\title{Privately detecting changes in unknown distributions }



\author{Rachel Cummings\footnotemark[1] \and Sara Krehbiel\footnotemark[2] \and Yuliia Lut\footnotemark[1] \and Wanrong Zhang\footnotemark[1]}

\renewcommand{\thefootnote}{\fnsymbol{footnote}}
\footnotetext[1]{School of Industrial and Systems Engineering, Georgia Institute of Technology. Email: \texttt{\{rachelc, yuliia.lut, wanrongz\}@gatech.edu}.  R.C. supported in part by a Mozilla Research Grant, a Google Research Fellowship, and NSF grant CNS-1850187. Y.L. and W.Z. supported in part by a Mozilla Research Grant, NSF grant CNS-1850187, and two ARC-TRIAD Fellowships from the Georgia Institute of Technology.  Much of this work was completed while R.C., Y.L., and W.Z. were visiting the Simons Institute for the Theory of Computing.}
\footnotetext[2]{Department of Mathematics and Computer Science, Santa Clara University. Email: \texttt{skrehbiel@scu.edu}. Supported in part by a Mozilla Research Grant. Much of this work was completed while S.K. was a faculty member at University of Richmond and visiting the Stanford Graduate School of Business.}


\maketitle

\renewcommand{\thefootnote}{\arabic{footnote}}

\begin{abstract}
The change-point detection problem seeks to identify distributional changes in streams of data. Increasingly, tools for change-point detection are applied in settings where data may be highly sensitive and formal privacy guarantees are required, such as identifying disease outbreaks based on hospital records, or IoT devices detecting activity within a home. Differential privacy has emerged as a powerful technique for enabling data analysis while preventing information leakage about individuals. Much of the prior work on change-point detection---including the only private algorithms for this problem---requires complete knowledge of the pre-change and post-change distributions. However, this assumption is not realistic for many practical applications of interest. This work develops differentially private algorithms for solving the change-point problem when the data distributions are unknown. Additionally, the data may be sampled from distributions that change smoothly over time, rather than fixed pre-change and post-change distributions.  We apply our algorithms to detect changes in the linear trends of such data streams.  Finally, we also provide experimental results to empirically validate the performance of our algorithms.
\end{abstract}

\newpage

\input{intro}
\input{prelims}


\input{offlinenew}

\input{online}

\input{app}

\input{simulations}



\bibliographystyle{alpha}
\bibliography{ref}


\end{document}

%% file: intro.tex
\section{Introduction}

The \emph{change-point detection problem} seeks to identify distributional changes in streams of data.  It models data points as initially being sampled from a pre-change distribution $P_0$, and then at an unknown change-point time $k^*$, data points switch to being sampled from a post-change distribution $P_1$.  The task is to quickly and accurately identify the change-point time $k^*$.  The change-point problem has been widely studied in theoretical statistics \cite{shewhart:1931,page:1954,shiryaev:1963,pollak:1987,mei:2008a} as well as practical applications including climatology \cite{Lund:2002}, econometrics \cite{Bai:Perron:2003}, and DNA analysis \cite{Zhang:Siegmund:2012}.  

Much of the previous work on change-point detection focused on the \emph{parametric} setting, where the distributions $P_0$ and $P_1$ are perfectly known to the analyst.  In this structured setting, the analyst could use algorithms tailored to details of these distributions, such as computing the maximum log-likehood estimator (MLE) of the change-point time. In this work, we consider the \emph{nonparametric} setting, where these distributions are unknown to the analyst.  This setting is closer to practice, as it removes the unrealistic assumption of perfect distributional knowledge. In practice, an analyst may only have sample access to the current (pre-change) distribution, and may wish to detect a change to any distribution that is sufficiently far from the current distribution without making specific parametric assumptions on the future (post-change) distribution. The nonparametric setting requires different test statistics, as common approaches like computing the MLE do not work without full knowledge of $P_0$ and $P_1$.

In many applications, change-point detection algorithms are applied to sensitive data, and may require formal privacy guarantees.  For example, the Center for Disease Control (CDC) may wish to analyze hospital records to detect disease outbreaks, or the Census Bureau may wish to analyze income records to detect changes in employment rates. We will use \emph{differential privacy} \cite{DMNS06} as our privacy notion, which has been well-established as the predominant privacy notion in theoretical computer science.  Informally, differential privacy bounds the effect of any individual's data in a computation, and ensures that very little can be inferred about an individual from seeing the output of a differentially private analysis. Differential privacy is achieved algorithmically by adding noise that scales with the \emph{sensitivity} of a computation, which is the maximum change in the function's value that can be caused by changing a single entry in the database.  High sensitivity analyses require large amounts of noise, which imply poor accuracy guarantees (see Section \ref{s.dpbackground} for more details).

 
Unfortunately, most nonparametric estimation procedures are not amenable to differential privacy. 
  Indeed, all prior work on private change-point detection has been in the parametric setting, where $P_0$ and $P_1$ are known \cite{CKM+18,canonne2018structure}. A standard approach in the nonparametric setting is to first estimate a parametric model, and then perform parametric change-point detection using the estimated model.  Common nonparametric estimation techniques include kernel methods and spline methods \cite{parzen1962estimation, rosenblatt1956remarks} or nonparametric regression \cite{azzalini1989use}. 
These methods are difficult to make private in part because of the complexity of finite sample error bounds combined with the effect of injecting additional noise for privacy. 
In contrast, simple rank-based statistics--which order samples by their value--have easy-to-analyze sensitivity. 


In this work, we estimate nonparametric change-points using the Mann-Whitney test \cite{wilcoxon1945individual,mann1947test}, which is a rank-based test statistic, presented formally in Section \ref{s.cpbackground}.  This test picks an index $k$ and measures the fraction of points before $k$ that are greater than points after $k$.  For the change-point problem, this statistic should be largest around the true change-point $k^*$, and smaller elsewhere (under mild non-degeneracy conditions on the pre- and post-change distributions).  Also note that this statistic simply computes pairwise comparisons of the observed data, and it does not require any additional knowledge of $P_0$ or $P_1$ beyond the assumption that a data point from $P_0$ is larger than a data point from $P_1$ with probability $>1/2$. The test statistic has sensitivity $O(1/n)$ for a database of size $n$, and is known to have lower sensitivity than most other test statistics for the same task \cite{mann1947test}.




\subsection{Our Results}

In this paper, we provide differentially private algorithms for accurate nonparametric change-point detection in both the offline and online settings. We also show how our results can be applied to settings where data are not sampled i.i.d., but are instead sampled from distributions changing smoothly over time.

In the offline case, the entire database $X=\{x_1,\ldots,x_n\}$ is given up front, and the analyst seeks to estimate the change-point with small additive error.  We use the Mann-Whitney rank-sum statistic and its extension to the change-point setting due to \cite{darkhovskh1976nonparametric}.  At every possible change-point time $k$, the test measures the fraction of points before $k$ that are greater than points after $k$, using statistic $V(k) = \frac{\sum_{j=k+1}^n\sum_{i=1}^k I(x_i>x_j)}{k(n-k)}$.  The test then outputs the index $\hat{k}$ that maximizes this statistic.  Even before adding privacy, we improve the best previously-known finite sample accuracy guarantees of this estimation procedure.  The previous non-private accuracy guarantee has $O(n^{2/3})$ additive error \cite{darkhovskh1976nonparametric}, whereas our Theorem \ref{nonprivate.acc} in Section \ref{nonprivate.offline} achieves $O(1)$ additive error.

With these improved accuracy bounds, we give Algorithm \ref{algo.offline}, \PNCPD\ in Section \ref{private.offline} to make this estimation procedure differentially private.  Our algorithm uses the \ReportMax\ framework of \cite{dwork2014algorithmic}. The \ReportMax\ algorithm takes in a collection of queries, computes a noisy answer to each query, and returns the index of the query with the largest noisy value.  We instantiate this framework with our test statistics $V(k)$ as queries, to privately select the argmax of the statistics. One challenge is ensuring that the test statistics $V(k)$ have low enough sensitivity that the additional noise required for privacy does not harm the estimation error by too much. We show that our \PNCPD\ algorithm is differentially private (Theorem \ref{thm.offpriv}) and has $O(\frac{1}{\eps^{1.01}})$ additive accuracy (Theorem \ref{private.acc}), meaning that adding privacy does not create any dependence on $n$ in the accuracy guarantee.

In the online case, the analyst starts with an initial database of size $n$, and receives a stream of additional data points, arriving online. The analyst's goal here is to accurately estimate the change-point quickly after it occurs.  This is a more challenging setting because the analyst will have very little post-change data if they want to detect changes quickly.  In this setting, we give Algorithm \ref{algo.online}, \OnlinePNCPD\ in Section \ref{s.online}.  This algorithm uses the \AboveThreshold\ framework of \cite{DNRRV09,DNPR10}.  The \AboveThreshold\ algorithm takes in a potentially unbounded stream of queries, compares the answer of each query to a fixed noisy threshold, and halts when it finds a noisy answer that exceeds the noisy threshold.  Our algorithm computes the test statistic $V(k)$ for the middle index $k$ of each sliding window of the last $n$ data points.  Once the algorithm finds a window with a high enough test statistic, it waits for enough additional data points to meet the requirements of our {\em offline} algorithm \PNCPD\ for accuracy, and then calls \PNCPD\ on the $n$ most recent data points to estimate the change-point time. One technical challenge in the online setting is finding a threshold that is high enough to prevent false positives before a change occurs, and low enough that a true change will trigger a call to the offline algorithm.  We show that our \OnlinePNCPD\ algorithm is differentially private (Theorem \ref{thm.onlinepriv}) and has $O(\log n)$ additive error (Theorem \ref{online.acc}).

In Section \ref{s.drift} we apply our results to privately solve the problem of \emph{drift change detection}, where points are not sampled i.i.d.~pre- and post-change, but instead are sampled from smoothly changing distributions whose means are shifting linearly with respect to time, and the linear drift parameter changes at an unknown change-time $k^*$. We show how to reduce an instance of the drift change detection problem with non-i.i.d.~samples to an instance of the change-point detection problem to which our algorithms can be applied.  We show in Corollary \ref{smooth.theorem}
 that our algorithms also provide differential privacy and accurate estimation for the drift change detection problem. We also suggest extensions of this reduction technique so that our algorithms may also be applied for non-linear drift change detection for other smoothly changing distributions that exhibit sufficient structure.
 
In Section \ref{s.sim} we report experimental results that empirically validate our theoretical results.  We start by applying our \PNCPD\ algorithm to a real-world dataset of stock price time-series data that appear by visual inspection to contain a change-point, and find that our algorithm does find the correct change-point with minimal loss in accuracy, even for small $\epsilon$ values.  We then apply our \PNCPD\ algorithm to simulated datasets sampled from Gaussian distributions, varying the parameters corresponding to the size of the distributional change, the location of the change-point in the dataset, and $\epsilon$.  We also perform simulations for our application to drift change detection by simulating data points drawn from the drift change model, performing the reduction described in Section \ref{s.drift}, and applying our \PNCPD\ algorithm to the resulting dataset.  Lastly we apply our \OnlinePNCPD\ algorithm to streaming simulated datasets drawn from Gaussian distributions, again varying the parameters that correspond to the size of the distributional change, the location of the change-point in the dataset, and $\epsilon$.  In all cases, the empirical accuracy corresponds qualitatively to what our theoretical results predict.





\subsection{Related work}

Change-point detection is a canonical problem in statistics that has been studied for nearly a century; selected results include \cite{shewhart:1931,page:1954,shiryaev:1963,roberts:1966,lorden:1971,pollak:1985, pollak:1987, moustakides:1986,lai:1995, lai:2001,kulldorff:2001,mei:2006a, mei:2008a, mei:2010,chan:2017}. The problem originally arose from industrial quality control, and has since been applied in a wide variety of other contexts including climatology \cite{Lund:2002}, econometrics \cite{Bai:Perron:2003}, and DNA analysis \cite{Zhang:Siegmund:2012}.  In the parametric setting where pre-change and post-change distributions $P_0$ and $P_1$ are perfectly known, the Cumulative Sum (CUSUM) procedure \cite{page:1954} is among the most commonly used algorithms for solving the change-point detection problem.  It follows the generalized log-likelihood ratio principle, calculating $\ell(k)=\sum_{i=k}^n \log \frac{P_1(x_i)}{P_0(x_i)}$ for each $k\in[n]$, and declaring that a change occurs if and only if $\ell(\hat k)\ge T$ for MLE $\hat k = \argmax_k \ \ell(k)$ and appropriate threshold $T>0$.  Nonparametric change-point detection has also been well-studied in the statistics literature \cite{darkhovskh1976nonparametric,carlstein1988nonparametric, bhattacharyya1968nonparametric}, and requires different test statistics that do not rely on exact knowledge of the distributions $P_0$ and $P_1$.

The only two prior works on differentially private change-point detection \cite{CKM+18,canonne2018structure} both considered the parametric setting and employed differentially private variants of the CUSUM procedure and the change-point MLE underlying it. \cite{CKM+18} directly privatized non-private procedures for the offline and online settings
. \cite{canonne2018structure} gave private change-point detection as an instantiation of a solution to the more general problem of private hypothesis testing, partitioning time series data into batches of size equal to the sample complexity of the hypothesis testing problem, and then outputs the batch number most consistent with a change-point. Both works assumed that the pre- and post-distributions were fully known in advance. 

In our nonparametric setting, we use the Mann-Whitney test \cite{wilcoxon1945individual,mann1947test} instead of the MLE that the CUSUM procedure is built on.  The Mann-Whitney test was originally proposed as a rank-based nonparametric two-sample test, to test whether two samples were drawn from the same distribution using the null hypothesis that after randomly selecting one point from each sample, each point is equally likely to be the larger of the two. It was extended to the change-point setting by \cite{darkhovskh1976nonparametric}, for testing whether samples from before and after the hypothesized change-point were drawn from the same distribution.  Given a database $X=\{x_1, \ldots, x_n\}$, for each possible change-point $k$, the test statistic $V(k) = \frac{\sum_{j=k+1}^n\sum_{i=1}^k I(x_i>x_j)}{k(n-k)}$ counts the proportion of index pairs $(i,j)$ with $i\le k<j$ for which $x_i>x_j$.  This is a nonparametric test because it does not require any additional knowledge of the distributions from which data are drawn.  Additionally, the Mann-Whitney test is known to be more efficient \cite{gibbons2011nonparametric} and have lower sensitivity \cite{mann1947test} than most other test statistics for the same task, including the Wald statistic \cite{wald1940test} and the Kolmogorov-Smirnov test \cite{lilliefors1967kolmogorov}. Differentially private versions of related test statistics have been used in recent unpublished work in the context of hypothesis testing, but they have not been applied to the change-point problem \cite{CKS+18,CKS+19}.

Although the current paper largely follows the same structure as \cite{CKM+18} for privatizing the change-point procedure, the analysis of the algorithm is vastly different, due to new challenges introduced by the nonparametric setting. Most test statistics for nonparametric estimation have high sensitivity, and therefore require large amounts of noise to be added to satisfy differential privacy.  This means that off-the-shelf applications of nonparametric test statistics to the differentially private change-point framework of \cite{CKM+18} would result in high error.  Indeed, even with our use of the Mann-Whitney test statistic which was chosen for its low sensitivity, an immediate application of the best known finite-sample accuracy bounds \cite{darkhovskh1976nonparametric} yielded additive error $O(n^{2/3})$ in the offline setting for databases of size $n$.  To achieve our much tighter $O(\epsilon^{-1.01})$ error bounds required a new analysis.

%% file: prelims.tex
\section{Preliminaries}\label{s.prelims}
This section provides the necessary background for interpreting our results for the problem of private nonparametric change-point detection.  Section \ref{s.cpbackground} defines the nonparametric change-point detection problem, Section \ref{s.dpbackground} describes the differentially private tools that will be brought to bear, and Section \ref{s.concentration} gives the concentration inequality which will be used in our proofs.

\subsection{Change-point background}\label{s.cpbackground}

Let $X=\{x_1,\dots,x_n\}$ be $n$ real-valued data points.  The \emph{change-point detection problem} is parametrized by two distributions, $P_0$ and $P_1$.  The data points in $X$ are hypothesized to initially be sampled i.i.d.~from $P_0$, but at some unknown change time $k^\ast \in [n]$, an event may occur (e.g., epidemic disease outbreak) and change the underlying distribution to $P_1$.  The goal of a data analyst is to announce that a change has occurred as quickly as possible after $k^\ast$.  Since the $x_i$ may be sensitive information---such as individuals' medical information or behaviors inside their home---the analyst will wish to announce the change-point time in a privacy-preserving manner.


In the standard non-private offline change-point literature, the analyst wants to test the null hypothesis $H_0:k^*=n$, where $x_1,\dots,x_n\sim_{\text{iid}}P_0$, against the composite alternate hypothesis $H_1:k^*<n$, where $x_1,\dots,x_{k^*}\sim_{\text{iid}} P_0$ and $x_{k^*+1},\dots,x_n\sim_{\text{iid}} P_1$. If $P_1$ and $P_0$ are known, the log-likelihood ratio of $k^*=\infty$ against $k^*=k$ will be given by 
\begin{equation*}\ell(k,X)=\sum_{i=k+1}^n \log \frac{P_1(x_i)}{P_0(x_i)}.\end{equation*}
The maximum likelihood estimator (MLE) of the change time $k^*$ is given by
$\argmax_{k\in[n]} \ell(k,X).$ However, note that to perform this test, the analyst must have complete knowledge of distributions $P_0$ and $P_1$ to compute the log-likelihood ratio.

In this paper, we consider the situation that we do not know both the pre-change distribution and the post-change distribution. We require no knowledge of the pre- and post- change distributions, and assume only that the probability of an observation from $P_0$ exceeding an observation from $P_1$ is different than the probability of an observation from $P_1$ exceeding an observation from $P_0$, which is necessary for technical reasons. 
The Mann-Whitney test \cite{wilcoxon1945individual} is a commonly used nonparametric test of the null hypothesis that it is equally likely that a randomly selected value from one sample will be less than or greater than a randomly selected value from a second sample.  \cite{darkhovskh1976nonparametric} proposed a modification of the Mann-Whitney test to solve the change-point estimation problem. For each possible change-point $k$, a test statistic counting the proportion of index pairs $(i,j)$ with $i\le k, j>k$ for which $x_i>x_j$ is calculated as follows:
\begin{equation}
V(k,X) = \frac{\sum_{j=k+1}^n\sum_{i=1}^k I(x_i>x_j)}{k(n-k)} \label{eq.Vk}
\end{equation}
For data $X$ drawn according to the change-point model with distributions $P_0,P_1$, this statistic is largest or smallest in expectation at the true change-point $k^*$ depending on the value $a=\Pr_{x_0\sim P_0,x_1\sim P_1}[x_0>x_1]$. If $a>1/2$, we estimate the change-point by taking the arg max of the Mann-Whitney statistics; otherwise we take the arg min.
When $X$ is clear from context, we will simply write $V(k)$. The estimator $\hat k$ is understood to denote the argmax or argmin of $V(k)$ depending on whether $a>1/2$.


We will measure the additive error  of our estimations of the true change-point as follows. 

\begin{definition}[$(\alpha, \beta)$-accuracy]
	A change-point detection algorithm that produces a change-point estimator~$\tilde{k}$   is \emph{$(\alpha,\beta)$-accurate} if $\Pr[|\tilde{k}-k^\ast|>\alpha] \le \beta$, where the probability is taken over randomness of the data $X$ sampled according to the change-point model with true change-point $k^*$ and (possibly) the randomness of the algorithm.
\end{definition}

\subsection{Differential privacy background}\label{s.dpbackground}

Differential privacy bounds the maximum amount that a single data entry can affect analysis performed on the database.  Two databases $X,X'$ are \emph{neighboring} if they differ in at most one entry.

\begin{definition}[Differential Privacy \cite{DMNS06}]\label{def.dp}
	An algorithm $\mathcal{M}: \mathbb{R}^n \rightarrow \mathcal{R}$ is \emph{$\epsilon$-differentially private} if for every pair of neighboring databases $X,X' \in \mathbb{R}^n$, and for every subset of possible outputs $\mathcal{S} \subseteq \mathcal{R}$,
	\[ \Pr[\mathcal{M}(X) \in \mathcal{S}] \leq \exp(\epsilon)\Pr[\mathcal{M}(X') \in \mathcal{S}]. \]
\end{definition} 

One common technique for achieving differential privacy is by adding Laplace noise.  The \emph{Laplace distribution} with scale $b$ is the distribution with probability density function: $\Lap(x|b) = \frac{1}{2b} \exp\left(-\frac{|x|}{b}\right)$.  We will write $\Lap(b)$ to denote the Laplace distribution with scale $b$, or (with a slight abuse of notation) to denote a random variable sampled from $\Lap(b)$.
The \emph{sensitivity} of a function or query $f$ 
is defined as 
$\Delta (f) = \max_{\text{neighbors } X, X'} | f(X) - f(X') |$, and it determines the scale of noise that must be added to satisfy differential privacy.  The Laplace Mechanism of \cite{DMNS06} takes in a function $f$, database $X$, and privacy parameter $\epsilon$, and outputs $f(X) + \Lap(\Delta (f)/\epsilon)$. 

One helpful property of differential privacy is that it \emph{composes}, meaning that the privacy parameter degrades gracefully as additional computations are performed on the same database.  

\begin{theorem}[Basic Composition \cite{DMNS06}]\label{thm.basic}
Let $\mathcal{M}_1$ be an algorithm that is $\epsilon_1$-differentially private, and let $\mathcal{M}_2$ be an algorithm that is $\epsilon_2$-differentially private.  Then their composition $(\mathcal{M}_1,\mathcal{M}_2)$ is $(\eps_1 + \eps_2)$-differentially private.
\end{theorem}

Our algorithms rely on the existing differentially private algorithms \ReportMax~\cite{dwork2014algorithmic}. The \ReportMax\ algorithm takes in a collection of queries, computes a noisy answer to each query, and returns the index of the query with the largest noisy value.  We use this as the framework for  our offline private nonparametric change-point detector \PNCPD\ in Section \ref{offline} to privately select the time $k$ with the highest Mann-Whitney statistics $V(k)$.

{\centering
	\begin{minipage}{\linewidth}
		\begin{algorithm}[H]
			\caption{Report Noisy Max: \ReportMax($X, \Delta, \{f_1, \ldots, f_m\}, \epsilon$)}
			\begin{algorithmic}
				\State \textbf{Input:} database $X$, set of queries $\{f_1, \ldots, f_m\}$ each with sensitivity $\Delta$, privacy parameter $\epsilon$
				\For {$i=1,\ldots,m$}
				\State Compute $f_i(X)$
				\State Sample $Z_i \sim \Lap(\frac{\Delta}{\epsilon})$
				\EndFor
				\State Output $i^* =\underset{i \in [m]}{\mathrm{argmax}} \left(f_i(X) + Z_{i} \right)$
			\end{algorithmic}
		\end{algorithm}
	\end{minipage}
}

\begin{theorem}[\cite{dwork2014algorithmic}]
	\ReportMax\ is $\epsilon$-differentially private. 
\end{theorem}

The \AboveThreshold\ algorithm of \cite{DNRRV09,DNPR10}, refined to its current form by \cite{dwork2014algorithmic}, takes in a potentially unbounded stream of queries, compares the answer of each query to a fixed noisy threshold, and halts when it finds a noisy answer that exceeds the noisy threshold.  We use this algorithm as a framework for our online private nonparametric change-point detector \OnlinePNCPD\ in Section \ref{s.online} when new data points arrive online in a streaming fashion.

{\centering
	\begin{minipage}{\linewidth}
		\begin{algorithm}[H]
			\caption{Above Noisy Threshold: \AboveThreshold($X, \Delta, \{f_1, f_2, \ldots \}, T, \epsilon$) }
			\begin{algorithmic}
				\State \textbf{Input:} database $X$, stream of queries $\{f_1, f_2, \ldots \}$ each with sensitivity $\Delta$, threshold $T$, privacy parameter $\epsilon$
				\State Let $\hat{T} = T + \Lap(\frac{2\Delta}{\epsilon})$
				\For {each query $i$}
				\State Let $Z_i \sim \Lap(\frac{4\Delta}{\epsilon})$
				\If {$f_i(X)+Z_i>\hat{T}$}
				\State Output $a_i=\top$
				\State Halt
				\Else
				\State Output $a_i=\bot$
				\EndIf
				\EndFor
			\end{algorithmic}\label{alg.ant}
		\end{algorithm}
	\end{minipage}
	}
	
	\begin{theorem}[\cite{DNRRV09}]\label{thm.antpriv}
		\AboveThreshold\ is $\epsilon$-differentially private. 
	\end{theorem}

\begin{theorem}[\cite{DNRRV09}]\label{thm.atacc}
	For any sequence of $m$ queries $f_1,\ldots, f_m$ with sensitivity $\Delta$ such that $|\{i<m:f_i(X)\ge T-\alpha\}|=0$, \\AboveThreshold\ outputs with probability at least $1-\beta$ a stream of $a_1, \ldots, a_m \in \{\top, \bot\}$ such that $a_i=\bot$ for every $i\in[m]$ with $f(i)<T-\alpha$ and $a_i=\top$ for every $i\in[m]$ with $f(i)>T+\alpha$ as long as
	\begin{align*}\alpha\ge \frac{8\Delta\log (2m/\beta)}{\epsilon}. \end{align*}
\end{theorem}

\subsection{Concentration inequalities}\label{s.concentration}

Our proofs will also use the following concentration inequality.

\begin{theorem}[McDiarmid \cite{McD89}] \label{lem.mcd}
	Define the discrete derivatives of the function $f(X_1,\ldots,X_n)$ of independent random variables $X_1,\ldots, X_n$ as
	\begin{equation}\label{dd_def}
	D_if(x):=\sup_{z}f(x_1,\ldots,x_{i-1},z,x_{i+1},\ldots,x_n)-\inf_{z}f(x_1,\ldots,x_{i-1},z,x_{i+1},\ldots,x_n).
	\end{equation}
	Then for $X_1,\ldots,X_n$ independent, $f(X_1,\ldots,X_n)$ is subgaussian with variance proxy $\frac{1}{4}\sum_{i=1}^n||D_if||^2_{\infty}.$ In particular,
	$$\Pr[f(X_1,\ldots,X_n)-\E [f(X_1,\ldots,X_n)]\ge t]\le \exp(-\frac{2t^2}{\sum_{i=1}^n ||D_if||^2_{\infty}}).$$
\end{theorem} 

%

%% file: offlinenew.tex
\section{Offline private nonparametric change-point detection}\label{offline}

In this section, we give an offline private algorithm for change-point detection when the pre- and post-change distributions are unknown. In Section \ref{nonprivate.offline}, we first offer the finite sample accuracy guarantee for the non-private nonparametric algorithm given by $\hat k = \argmax\ V(k)$ for the test statistic $V(k)$ given in Equation \eqref{eq.Vk}, which will serve as the baseline for evaluating the utility of our private algorithm. Then in Section \ref{private.offline} we present our private algorithm, and give privacy and accuracy guarantees.

\subsection{Finite sample accuracy guarantee for the non-private nonparametric estimator}\label{nonprivate.offline}

In this section, we provide error bounds for the non-private nonparametric change-point estimator 
when the data are drawn from two unknown distributions $P_0$, $P_1$ with true change-point $k^*\in\set{\ceil{\gamma n},\dots, \floor{(1-\gamma)n}}$, for some known $\gamma<1/2$.  This $\gamma$ bounds away from the change-point occurring too early or too late in the sample, and is necessary to ensure sufficient number of samples from both the pre-change and post-change distributions. Without loss of generality, we assume that $a := \Pr_{x_0\sim P_0,x_1\sim P_1}[x_0>x_1]>1/2$.

For the non-private task, we use the following estimation procedure of \cite{darkhovskh1976nonparametric}, which calculates the estimated change-point $\hat k$ as the argmax of $V(k)$ over all $k$ in the range permitted by $\gamma$:
\begin{align*}
\hat k = \argmax_{k\in\set{\ceil{\gamma n},\dots, \floor{(1-\gamma)n}}} V(k),
\end{align*}
for test statistic $V(k)$ defined in Equation \eqref{eq.Vk}.
We show in Theorem \ref{nonprivate.acc} that the additive error of this procedure is constant with respect to the sample size $n$.

Our result is much tighter that the previously known finite-sample accuracy result in \cite{darkhovskh1976nonparametric}, which gave an estimation error bound of $O(n^{2/3})$. 
This sublinear result comes from a connection between the accuracy and the maximal deviation of  $V(k)$ from the expected value over $[\gamma n, (1-\gamma)n]$ . To bound the maximal deviation, \cite{darkhovskh1976nonparametric} first analyzed the variance approximation of $V(k)$ to bound the deviation for a single point $k$. Then they utilized a Lipschitz property to partition $[\gamma n, (1-\gamma)n]$ to small intervals, and took a union bound over these intervals to yield a high probability guarantee.       
In contrast, we better leverage the connection between $V(k)$ and $V(k^*)$ for improved accuracy and a simplified proof.
  At a high level, we show that the expectation of $V(k)$ is single-peaked around $k^*$, and $V(k)-V(k^*)$ is subgaussian. We carefully analyze the discrete derivative as a function of $\abs{k^*-k}$, $\gamma$, and $n$ to use a concentration bound yielding our constant error result. 

\begin{theorem}\label{nonprivate.acc}
	For data $X=\set{x_1,\dots,x_n}$ drawn according to the change-point model with any distributions $P_0,P_1$ with $a=\Pr_{x\sim P_0,y\sim P_1}[x>y]>1/2$, constraint $\gamma\in(0,1/2)$, and change-point $k^*\in \set{\ceil{\gamma n},\dots, \floor{(1-\gamma)n}}$, we have that the estimator 
	$$\hat k = \argmax_{k\in\set{\ceil{\gamma n},\dots, \floor{(1-\gamma)n}}} \frac{\sum_{j=k+1}^n\sum_{i=1}^k I(x_i>x_j)}{k(n-k)}$$ is $(\alpha,\beta)$-accurate for any $\beta>0$ and 
	$$\alpha= C\cdot \left(\frac{1}{\gamma^4 (a-1/2)^2}\right)^c\cdot \log \frac 1 \beta $$
	for any constant $c>1$ and some constant $C>0$ depending on $c$.
	
If $a<1/2$ we achieve the same error bound using $\hat k = \argmin \frac{\sum_{j=k+1}^n\sum_{i=1}^k I(x_i>x_j)}{k(n-k)}$.
	\end{theorem}

\begin{proof}
	We will show that for $\hat k = \argmax\  V(k)$ and $\alpha$ as in the theorem statement, 
	$$\Pr[\abs{\hat k - k^*}>\alpha]\le \sum_{k:\abs{k-k^*}>\alpha}\Pr[V(k)>V(k^*)]\le \beta.$$
	To do this, we fix any $k\in\set{\ceil{\gamma n},\dots,\floor{(1-\gamma) n}}$ and show that $f(X)=V(k)-V(k^*)$ is subgaussian. In particular, 
	for $k$ at least $\alpha$ away from $k^*$, the expectation of $V(k^*)-V(k)$ is sufficiently large and its discrete derivative is sufficiently small  that the probability of $V(k)>V(k^*)$ can be tightly bounded as a function of $\alpha$ by application of Theorem~\ref{lem.mcd}.
	
	First we give a lower bound the difference in expectation of $V(k^*)$ and $V(k)$. Observe that 
	\begin{align*}
	\E[V(k)] &= \frac{\sum_{i\le k,j>k} \Pr[x_i>x_j]}{k(n-k)}\\
	&= \begin{cases}
	\frac{\frac 1 2 (k^*-k) + a(n-k^*)}{n-k} & k\le k^* \\
	\frac{ak^* + \frac 1 2 (k-k^*)}{k} & k>k^*
	\end{cases},
	\end{align*}
	achieving its maximum at $\E[V(k^*)]=a$. Therefore, we can bound
	\begin{align}
	\E[V(k^*)-V(k)] &= \begin{cases}
	(a-\frac 1 2)\frac{k^*-k}{n-k} & k\le k^* \\
	(a-\frac 1 2)\frac{k-k^*}{k} & k>k^*
	\end{cases} \notag \\
	&\ge (a-\frac 1 2)\frac{\abs{k^*-k}}{n}. \label{eq.expVk}
	\end{align}

In the following bounds on the discrete derivative of $f(X)=V(k)-V(k^*)$, we will make use of the fact that $f$ can be written as:
\begin{align*}
f(X) =&  \frac{\sum_{j=k+1}^n\sum_{i=1}^{k}I(x_i>x_j)}{k(n-k)}-\frac{\sum_{j=k^*+1}^n\sum_{i=1}^{k^*}I(x_i>x_j)}{k^*(n-k^*)} \\
=& \left( \frac{1}{k(n-k)}-\frac{1}{k^*(n-k^*)}\right) \left(\sum_{\substack{i=1,\dots,k, \\ j=k+1,\dots,n}}I(x_i>x_j) \right)\\
&+\frac{1}{k^*(n-k^*)} \left( \sum_{\substack{i=1,\dots,k,\\j=k+1,\dots,n}}I(x_i>x_j)-\sum_{\substack{i=1,\dots,k^*,\\j=k^*+1,\dots,n}}I(x_i>x_j)\right)
\end{align*}


We bound the discrete derivative $D_i f$ separately for $i\le\min\set{k,k^*}$, $i\in(\min\set{k,k^*},\max\set{k,k^*}]$, and $i> \max\set{k,k^*}$. When $x_i$ changes arbitrarily for $i\le \min\set{k,k^*}$, we note that $\sum_{j=k+1}^n I(x_i>x_j)$ can change by at most $\pm(n-k)$ and $\sum_{j=k+1}^{k^*} I(x_i>x_j)$ can change by at most $\pm (k^*-k)$. These counts are normalized in $f$, and the normalization ensures this former count contributes at most $\frac{\abs{k^*-k}}{k^*(n-k^*)}+\frac{\abs{k^*-k}}{kk^*}$ to the discrete derivative.  We bound the discrete derivative for $i\le \min\set{k,k^*}$ as follows:
\begin{align*}
D_i f &\le \abs{\frac{1}{k(n-k)}-\frac{1}{k^*(n-k^*)}} \left( n-k \right)+\frac{\abs{k^*-k}}{k^*(n-k^*)} 
\\
&=  \abs{\frac{1}{k}-\frac{n-k}{k^*(n-k^*)}}+\frac{\abs{k^*-k}}{k^*(n-k^*)}  \\
&=\abs{-\frac{\abs{k-k^*}}{k^*k}+\frac{\abs{k-k^*}}{k^*(n-k^*)}}+\frac{\abs{k-k^*}}{k^*(n-k^*)}   \\
&\le \frac{\abs{k-k^*}}{\gamma^2 n^2}+\frac{2\abs{k-k^*}}{\gamma(1-\gamma)n^2}   \\
&\le \frac{3\abs{k-k^*}}{\gamma^2 n^2}
\end{align*}

We bound the discrete derivative for $i>\max\set{k,k^*}$ similarly, noting that an arbitrary change in $x_i$ changes $\sum_{i'=1}^k I(x_{i'}>x_i)$ by at most $\pm k$ and $\sum_{i'=k^*+1}^k I(x_{i'}>x_i)$ by at most $\pm (k-k^*)$:
\begin{align*}
D_i f &\le \abs{\frac{1}{k(n-k)}-\frac{1}{k^*(n-k^*)}} \cdot k+\frac{\abs{k^*-k}}{k^*(n-k^*)} 
\\
&=  \abs{\frac{1}{n-k}-\frac{k}{k^*(n-k^*)}}+\frac{\abs{k^*-k}}{k^*(n-k^*)}  \\
&=\abs{-\frac{\abs{k-k^*}}{(n-k^*)(n-k)}+\frac{\abs{k-k^*}}{k^*(n-k^*)}}+\frac{\abs{k-k^*}}{k^*(n-k^*)}  \\
&\le \frac{\abs{k-k^*}}{\gamma^2 n^2}+\frac{2\abs{k-k^*}}{\gamma(1-\gamma)n^2}  \\
&\le \frac{3\abs{k-k^*}}{\gamma^2 n^2}
\end{align*}

Finally we bound the discrete derivative for $\min\set{k,k^*} < i \le \max\set{k,k^*}$. To do this, we note that the first summation in $f$ changes by $k$ if $k<k^*$ or $n-k$ if $k> k^*$, and the difference of summations in the second term changes by at most $n-(k+k^*)$ in either case. Then we achieve our bound as follows:
\begin{align*}
D_i f &\le \abs{\frac{1}{k(n-k)}-\frac{1}{k^*(n-k^*)}}\cdot \max\set{k,n-k}+\frac{n-(k^*+k)}{k^*(n-k^*)} \\
&\le \frac{\abs{k-k^*}}{\gamma^2 n^2}+\frac{n}{\gamma(1-\gamma)n^2} \\
&\le \frac 2 {\gamma^2 n}
\end{align*}

Then since $D_i f$ is finite for each $i$, we have that $f$ is subgaussian with variance proxy as follows:
\begin{align*}
\frac 1 4 \sum_{i=1}^n (D_i f)^2 
&\le \frac {n-\abs{k^*-k}}4 \cdot \frac{9\abs{k-k^*}^2}{\gamma^4 n^4} + \frac{\abs{k^*-k}}4\left(\frac{\abs{k-k^*}}{\gamma^2 n^2} + \frac{1}{\gamma(1-\gamma)n}\right)^2 \\
&\le \frac{9\abs{k-k^*}^2}{4\gamma^4 n^3} + \frac{\abs{k^*-k}}{\gamma^4 n^2} \\
&\le \frac{13\abs{k^*-k}}{4\gamma^4 n^2} 
\end{align*}

We can now bound the probability of outputting any particular $k=\ceil{\gamma n},\dots,\floor{(1-\gamma)n}$ as a function of $\abs{k-k^*}$ by applying Theorem~\ref{lem.mcd}, recalling our bound on $\E[V(k^*)-V(k)]$ from Equation \eqref{eq.expVk}.
\begin{align*}
\Pr[V(k)>V(k^*)] 
=&\Pr\left[V(k)-V(k^*)-\E [V(k)-V(k^*)]>\E[V(k^*)-V(k)]\right] \notag \\
\le&\Pr\left[V(k)-V(k^*)-\E [V(k)-V(k^*)]>(a-\frac{1}{2})\frac{|k-k^*|}{n}\right] \notag \\
\le& \exp(-\frac {2 \gamma^4}{13}(a-\frac{1}{2})^2|k-k^*|). 
\end{align*}

We complete the proof by bounding the probability of any incorrect $\hat k$ such that $|\hat k-k^*|>\alpha$ by~$\beta$. 

\begin{align}
\Pr[\abs{\hat k - k^*}>\alpha] &\le 2\sum_{|k-k^*|=\alpha}^n \exp(-\frac{2\gamma^4}{13}(a-\frac 1 2)^2\abs{k-k^*}) \notag \\
&\le \frac{2\exp(-\frac{2\gamma^4}{13}(a-\frac 1 2)^2\alpha)}{1-\exp(-\frac{2\gamma^4}{13}(a-\frac 1 2)^2)} \notag \\
&\le \beta \notag
\end{align}
Rearranging shows that our accuracy result will hold for 
\begin{align*}
\alpha
&\ge\frac{13}{2\gamma^4(a-1/2)^2}\left( \log\frac{2}{\beta}+\log\frac{1}{1-\exp(-\frac{2\gamma^4}{13}(a-\frac 1 2)^2)}\right)  
\end{align*}

We achieve our final bound by simplifying the above expression as follows. We observe that  $\gamma<1/2, a<1$ implies $x=2\gamma^4(a-1/2)^2/13\le 1/416$, and for small $x$ we have $\log(1/(1-\exp(-x)))\le 2\log(1/x)$. For any $c>0$, we have $\log (1/x)\le C(1/x)^c$ for any $1/x\ge 416$ and $C\ge (\log 416)/(416^c)$, which can be applied to get our final bound.

\end{proof}


\subsection{Private offline algorithm}\label{private.offline}

We now give a differentially private version of the nonparametric estimation procedure of \cite{darkhovskh1976nonparametric}, in Algorithm \ref{algo.offline}.  Our algorithm uses \ReportMax\ as a private subroutine, instantiated with queries $V(k)$ to privately compute $\argmax\ V(k)$. We show that our algorithm is differentially private (Theorem \ref{thm.offpriv}) and produces an estimator with additive accuracy that is constant with respect to the sample size $n$ (Theorem \ref{private.acc}).

The crux of the privacy proof involves analyzing the sensitivity of the Mann-Whitney statistic to ensure that sufficient noise is added for the \ReportMax\ algorithm to maintain its privacy guarantees. The low sensitivity of this test statistic plays a critical role in requiring only small amounts of noise to preserve privacy.
The accuracy proof extends Theorem~\ref{nonprivate.acc} for the non-private estimator to incorporate the additional error due to the Laplace noise added for privacy. Since the event $V(k)>V(k^*)$ is less probable for $k$ that are further away from $k^*$, our analysis permits larger values of Laplace noise $Z_k$ for $k$ far from $k^*$, allowing privacy ``for free'' with respect to accuracy, for small constants $\eps$. 

\begin{algorithm}[H]
	\caption{Private Nonparametric Change-Point Detector: \PNCPD($X,\epsilon,\gamma$)}
	\begin{algorithmic}
		\State  \textbf{Input:} Database $X=\set{x_1,\dots,x_n}$, privacy parameter $\eps$, contraint parameter $\gamma$.
		\For{$k=\set{\ceil{\gamma n},\dots,\floor{(1-\gamma)n}}$} 
		\State Compute $ V(k) =\frac{1}{k(n-k)}\sum_{j=k+1}^n \sum_{i=1}^k {I}(x_i>x_j) $
		\State Sample $Z_k \sim \Lap(\frac{2}{\epsilon\gamma n})$
		\EndFor
		\State Output $\tilde{k} =\underset{k\in\set{\ceil{\gamma n},\dots,\floor{(1-\gamma)n}}}{\mathrm{argmax}} \{V(k)+ Z_{k} \}$
	\end{algorithmic}\label{algo.offline}
\end{algorithm}

\begin{theorem}\label{thm.offpriv}
	For arbitrary data $X=\set{x_1,\dots,x_n}$, privacy parameter $\eps>0$, and constraint $\gamma\in(0,1/2)$, \PNCPD($X,\eps,\gamma$) is $\epsilon$-differentially private.
\end{theorem}

\begin{proof} Privacy follows by instantiation of \ReportMax\ with queries $V(k)$ for $k\in\set{\ceil{\gamma n},\dots,\floor{(1-\gamma)n}}$, which have sensitivity $\Delta(V)=1/(\gamma n)$, with the observation that noise parameter $2\Delta(V)/\eps$ suffices for non-monotonic statistics. 
We include a proof for completeness.
	
	Fix any two neighboring databases $X,X'$ that differ on index $t$. For any $k\in\set{\ceil{\gamma n},\dots,\floor{(1-\gamma)n}}$, denote the respective rank statistics as $V(k)$ and $V'(k)$. By the definition of $V(k)$, we have
	\begin{align*}
	\abs{V(k)-V^\prime(k)}=
	\begin{cases}
	\frac{1}{k(n-k)}\abs{\sum_{j=k+1}^n\mathbb{I}(x_t>x_j)-\mathbb{I}(x_t'>x_j)} \le \frac{1}{k} \quad \text{if } t\le k \\
	\frac{1}{k(n-k)}\abs{\sum_{i=1}^k\mathbb{I}(x_i>x_t)-\mathbb{I}(x_i>x_t')} \le \frac{1}{n-k}\quad \text{if } t>k,
	\end{cases}
	\end{align*}	
	and it follows that $\Delta(V)=1/(\gamma n)$.

	Next, for a given $1 \le t \le n$, fix $Z_{-t}$, a draw from $[\Lap(2/\gamma\epsilon n)]^{n-1}$ used for all the noisy rank statistics values except the $l$th one. We will bound from above and below the ratio of the probabilities that the algorithm outputs $\tilde k = t$ on inputs $X$ and $X'$. 
	Define the minimum noisy value in order for $t$ to be selected with $X$:
	\begin{align*}
	Z^\ast_t &=\min \{{Z_t}: V(t)+Z_t>V(k)+Z_k \quad \forall k\neq t\} 
	\end{align*}
	
	For all $k\ne t$ we have
	$$V^\prime(t)+\Delta(V)+Z^\ast_t \ge V(t)+Z^\ast_t>V(k)+Z_k \ge V^\prime(k)-\Delta(V)+Z_k.$$	
	
	Hence, $Z_t^\prime \ge Z^\ast_t+ 2\Delta(V)$ ensures that the algorithm outputs $t$ on input ${X'}$, and the theorem follows from the following inequalities for any fixed $Z_{-t}$, with probabilities over the choice of $Z_t\sim \Lap\left(2/(\gamma\eps n)\right)$. 
	\begin{eqnarray*}
		\Pr[\tilde{k}=t \mid{X'}, Z_{-t}] \ge \Pr[ Z_i^\prime \ge Z^\ast_i+ 2\Delta(V)\mid Z_{-t}] \ge e^{-\epsilon}\Pr[ Z_t \ge Z^\ast_i\mid Z_{-t}] = e^{-\epsilon}\Pr[\tilde{k}=t \mid { X}, Z_{-i}]
	\end{eqnarray*}
	
\end{proof}

Next we provide an accuracy guarantee for our private algorithm \PNCPD\ when the data are drawn according to the change-point model. The first term in the error bound of Theorem~\ref{private.acc} comes from the randomness of the $n$ data points, and the second term is the additional cost that comes from the randomness of the sampled Laplace noises, which quantifies the cost of privacy. 
To ensure that the cost of privacy is as small as possible, we use \emph{$k$-specific thresholds} $t_k$ in the proof to optimize the trade-off between how much to tolerate the Laplace noise added for privacy versus the randomness of the data. As $\abs{k-k^*}$ increases, $V(k)$ is less likely to be close to $V(k^*)$, so we can allow more Laplace noise rather than set a universal tolerance for all $k$.

\begin{theorem}\label{private.acc}
	For data $X=\set{x_1,\dots,x_n}$ drawn according to the change-point model with any distributions $P_0,P_1$ with $a=\Pr_{x\sim P_0,y\sim P_1}[x>y]>1/2$, constraint $\gamma\in(0,1/2)$, change-point $k^*\in\set{\ceil{\gamma n},\dots,\floor{(1-\gamma)n}}$, and privacy parameter $\eps>0$, we have that \PNCPD($X,\eps,\gamma$) is $(\alpha,\beta)$-accurate for any $\beta>0$ and 
		\begin{equation*}
	\alpha=\max\{C_1\cdot \left(\frac{1}{\gamma^4 (a-1/2)^2}\right)^c\cdot \log \frac 1 \beta, C_2\cdot \left(\frac{1}{\epsilon\gamma (a-1/2)}\right)^c\cdot \log \frac 1 \beta \}, \label{private.alpha}
	\end{equation*}
	for any constant $c>1$ and some constants $C_1, C_2>0$ depending on $c$. 

\end{theorem}
As with our analysis of the non-private estimator, we can take the argmin and get the same error bounds (with $a-1/2$ replaced by $\abs{a-1/2}$) if $\Pr_{x\sim P_0,y\sim P_1}[x>y]<1/2$.

\begin{proof} We will show that for $\tilde k = \argmax\set{V(k)+Z_k}$ and $\alpha$ as in the theorem statement, 
$$\Pr[\abs{\tilde k-k^*}>\alpha]\le \sum_{k:\abs{k-k^*}> \alpha}\Pr[V(k)+Z_k>V(k^*)+Z_{k^*}] \le \beta$$
by showing that $V(k)-V(k^*)$ is subgaussian as in Theorem~\ref{nonprivate.acc}, and we will additionally show that the Laplace noise does not introduce too much additional error. 
For the algorithm to output an incorrect $\tilde{k}$, it must either be the case that the statistic $V(k)$ is nearly as large as $V(k^*)$ because of the randomness of the data points, or that $Z_k$ is much larger than $Z_{k^*}$. For each value of $k$, we choose a threshold $t_k$ increasing in $\abs{k-k^*}$ specifying how much to tolerate bad Laplace noise versus bad data, and we bound the probability that the algorithm outputs $k$ as follows:
	\begin{align}\label{bad.each}
	\Pr [V(k)+Z_k>V(k^*)+Z_{k^*}]\le \Pr[V(k^*)-V(k)<t_k]+\Pr[Z_k-Z_{k^*}>t_k]
	\end{align}
Setting  $t_k=(a-1/2)|k-k^*|/(2n)$, we can bound the first term as in Theorem \ref{nonprivate.acc} using Theorem \ref{lem.mcd} as follows:
	\begin{align*}
	\Pr[V(k)-V(k^*)>-t_k]&=\Pr\left[V(k)-V(k^*)-\E [V(k)-V(k^*)]>\left(a-\frac{1}{2}\right)\frac{|k-k^*|}{2n}\right] \notag\\
	&\le \exp\left(-\frac{\gamma^4\left(a-\frac{1}{2}\right)^2\abs{k-k^*}}{26}\right).
	\end{align*}
	We bound the second term of (\ref{bad.each}) by analyzing the Laplace noise directly. 
	\begin{align*}
	\Pr[Z_k-Z_{k^*}>t_k]&\le \Pr\left[2\abs{\Lap({2}/({\eps\gamma n}))}>\left(a-\frac{1}{2}\right)\frac{|k-k^*|}{2n}\right] \notag\\
	&\le \exp\left(-\frac{\left(a-\frac 1 2\right)\epsilon\gamma |k-k^*|}{8}\right) \label{eq.lap}
	\end{align*}
	
	We complete the proof by bounding the probability of any incorrect $\tilde k$ such that $\abs{\tilde k-k^*}>\alpha$ by $\beta$. 
	\begin{align*}
	\Pr\left[\abs{\tilde k-k^*}>\alpha\right] &\le 2\sum_{k:\abs{k-k^*}=\alpha}^n \exp\left(-\frac{\gamma^4\left(a-\frac{1}{2}\right)^2\abs{k-k^*}}{26}\right)+\exp\left(-\frac{\left(a-\frac 1 2\right)\epsilon\gamma |k-k^*|}{8}\right) \\
	&\le \frac{2\exp\left(-\frac{\gamma^4}{26}\left(a-\frac 1 2\right)^2\alpha\right)}{1-\exp\left(-\frac{\gamma^4}{26}\left(a-\frac 1 2\right)^2\right)} + \frac{2\exp\left(-\frac{\epsilon\gamma}{8}\left(a-\frac 1 2\right)\alpha\right)}{1-\exp\left(-\frac{\epsilon\gamma}{8}\left(a-\frac 1 2\right)\right)} \notag \notag \\
	&\le \beta
	\end{align*}
We bound each term above by $\beta/2$. Rearranging shows that our accuracy result will hold for 
\begin{align*} 
\alpha\ge \max\Bigg\{\frac{26}{\gamma^4\left(a-1/2\right)^2}\Bigg( \log\frac{4}{\beta}+\log\frac{1}{1-\exp\left(-\frac{\gamma^4}{26}\left(a-\frac 1 2\right)^2\right)}\Bigg), 
\\
\frac{8}{\epsilon\gamma\left(a-1/2\right)}\Bigg( \log\frac{4}{\beta}+\log\frac{1}{1-\exp\left(-\frac{\epsilon\gamma}{8}\left(a-\frac 1 2\right)\right)}\Bigg) \Bigg\}.
\end{align*}
We achieve our final bound by simplifying the above expression as follows. For the first term, we observe that  $\gamma<1/2, a<1$ implies $x=\gamma^4(a-1/2)^2/26\le 1/1664$, and for small $x$ we have $\log(1/(1-\exp(-x)))\le 2\log(1/x)$. For any $c>0$, we have $\log (1/x)\le C(1/x)^c$ for any $1/x\ge 1664$ and $C\ge (\log 1664)/(1664^c)$, which can be applied to get our final bound. For the second term, we observe that $x=\epsilon\gamma(a-\frac 1 2)/8 \le \epsilon/32$. When $\epsilon$ is small and the corresponding $x \le 4/5$, we have $\log(1/(1-\exp(-x)))\le 2\log(1/x)$, and for any $c>0$, we have $\log (1/x)\le C(1/x)^c$ for any $1/x\ge 5/4$ and $C\ge (\log 4/5)/((4/5)^c)$. When $\epsilon$ is large and the corresponding $x>4/5$, we have $\log(1/(1-\exp(-x)))\le \log2$, which can be incorporated into the constant in our final bound.

\end{proof}

%% file: online.tex
\section{Online change point detection}\label{s.online}

In this section, we show how to extend our results for change-point detection with unknown distributions to the \emph{online setting}, where the database $X$ is not given in advance, but instead data points arrive one-by-one.  We assume the analyst starts with a database of size $n$, and receives one new data point per unit time.  

Our algorithm uses the Above Noisy Threshold algorithm of \cite{DNRRV09,DNPR10} (\AboveThreshold, Algorithm \ref{alg.ant}) instantiated with queries of the Mann-Whitney test statistic $V(k)$ in the center of a sliding window of the most recent $n$ points. With each new data point $k>n$, the algorithm computes $V(k)$ for database $X=\set{x_{k-n/2+1},\ldots x_{k+n/2}}$, and compares this statistic against a noisy threshold for significance.  When this statistic is sufficiently high, the online algorithm calls the offline algorithm \PNCPD\ on this window to estimate $k^*$.  For simplicity in indexing and to avoid confusion with the notation of the previous section, we define $U(k) =V(k)$ when $V(k)$ is taken over database $X$ for each $k>n/2$.  Since the algorithm only checks for a change-point in the middle of the window, we assume that $k^* \geq n/2$ to ensure that the change-point does not occur too early to be detected.

We note that the offline subroutine \PNCPD\ assumes that a change point occurs sometime after the first $\gamma n$ and before the last $\gamma n$ of the $n$ data points on which it is called. We will show that for an appropriate choice of $T$, \OnlinePNCPD\ exceeds $\hat T$ for some $k$ such that $k^*\in[k,k+n/2]$. Therefore, by waiting for an additional $\gamma n$ data points, we ensure that the assumptions of \PNCPD\ are met as long as $\gamma<1/4$.

 \begin{algorithm}[H]
 \caption{Online Private Nonparametric Change-Point Detector: \OnlinePNCPD($X,n,\epsilon,\gamma,T$)}
 \begin{algorithmic}
\State  \textbf{Input:} Data stream $X$, starting size $n$ , privacy parameter $\eps$, constraint parameter $\gamma$, threshold $T$. 
\State Let $\hat T = T + \Lap\left(\frac{8}{\epsilon  n}\right)$
\For{ each new data point $x_{k+n/2}$, $k>n/2$} 
\State Compute $ U(k) =\frac{4}{n^2}\sum_{j=k+1}^{k+n/2} \sum_{i=k-n/2+1}^{k} I(x_i>x_j) $ 
\State Sample $Z_k \sim \Lap(\frac{16}{\epsilon n})$
\If{$U(k)+Z_k>\hat T$}
\State Wait for $\gamma n$ new data points to arrive
\State Output $\PNCPD\left(\set{x_{k-n/2+1+\gamma n}, \ldots x_{k+n/2+\gamma n}},\eps/2,\gamma\right)$ 
\State Halt
\EndIf
\EndFor
\end{algorithmic}\label{algo.online}
\end{algorithm}



Privacy follows immediately from the privacy guarantees of \AboveThreshold\ and \PNCPD.

\begin{theorem}\label{thm.onlinepriv}
For arbitrary data stream $X$ with starting size $n$, privacy parameter $\eps>0$, and constraint $\gamma\in(0,1/2)$, \OnlinePNCPD$(X,n,\eps,\gamma)$ is $\epsilon$-differentially private.
\end{theorem}

\begin{proof}
By Theorem \ref{thm.antpriv}, \AboveThreshold\ is $\epsilon$-differentially private, and by Theorem \ref{thm.offpriv}, the statistics $V(k)$ and $U(k)$ have sensitivity $2/ n$.  Also by Theorem \ref{thm.offpriv}, \PNCPD\ is $\epsilon$-differentially private. Thus the algorithm \OnlinePNCPD\ is simply \AboveThreshold\ instantiated with privacy parameter $\epsilon/2$, composed with \PNCPD\ also instantiated with privacy parameter $\epsilon/2$.  By Basic Composition (Theorem \ref{thm.basic}), \OnlinePNCPD$(X,n,\eps,\gamma)$ is $\epsilon$-differentially private.
\end{proof}

To give accuracy bounds on the performance of \OnlinePNCPD, we need to bound several sources of error.  First we need to set the threshold $T$ such that the algorithm will not raise a false alarm before the change-point occurs (i.e., control the false positive rate) and that the algorithm will not fail to raise an alarm on a window containing the true change-point (i.e., control the false negative rate).  This must be done taking into account the additional error from the private \AboveThreshold\ subroutine.  Finally, we can use the accuracy guarantees of \PNCPD\ to show that conditioned on calling a window that contains the true change-point, we are likely to output an estimator $\hat{k}$ that is close to the true change-point $k^*$. 

\begin{theorem}\label{online.acc}
For data stream $X$ with starting size $n$ drawn according to the change-point model with any distributions $P_0,P_1$ with $a=\Pr_{x\sim P_0,y\sim P_1}[x>y]>1/2$, constraint $\gamma\in(0,1/4)$, change-point $k^*\ge n/2$, privacy parameter $\eps>0$, and threshold $T\in[T_L,T_U]$ such that
	\begin{align*}
	    T_L&= \frac{1}{2}+ \sqrt{\frac{2}{n}\log\left( \frac{8(k^*-n/2)}{\beta} \right)}  + \frac{32\log( (k^*-n/2)/\beta)}{ n\epsilon} \\
	    T_U&= a - \sqrt{\frac{2}{n}\log\left( \frac{8}{\beta} \right)}-\frac{32\log(8(k^*-n/2)/\beta)}{ n\epsilon},
	\end{align*}
we have that \OnlinePNCPD($X,n,\eps,\gamma,T$) is $(\alpha,\beta)$-accurate for any $\beta>0$ and 
	\begin{equation*}
	\alpha=\max\left\{C_1\cdot \left(\frac{1}{\gamma^4 (a-1/2)^2}\right)^c\cdot \log \frac n \beta, C_2\cdot \left(\frac{1}{\epsilon\gamma (a-1/2)}\right)^c\cdot \log \frac n \beta \right\}, 
	\end{equation*}
	for any constant $c>1$ and some constants $C_1, C_2>0$ which depend only on $c$. 
\end{theorem}
%
\begin{proof}

First, we find an interval $[T_L, T_U]$ for the threshold $T$ that ensures that the algorithm neither calls \PNCPD\ before the true change-point has occurred nor fails to call \PNCPD\ on the window containing $k^*$ somewhere in the middle $(1-2\gamma)n$ data points.  
For now we will ignore the error from $\AboveThreshold$, and use $T_L', T_U'$ to denote the desired thresholds ignoring this additional source of noise. For ease of notation and reindexing, we define $U(k) = V(k)$ when $V(k)$ is computed over database $X=\set{x_{k-n/2+1}, \ldots x_{k+n/2}}$ for the Mann-Whitney test statistic $V(\cdot)$ as defined in Equation \eqref{eq.Vk}.

Thus we aim to find a range $[T_L^\prime, T_U^\prime]$  such that 
 \begin{align}
     & \Pr[U(k)>T_L^\prime | X_{k-n/2+1}, \ldots X_{k+n/2} \sim P_0] \le \frac{\beta}{8(k^* - n/2)},\label{T_l}\\
     & \Pr[U(k)<T_U^\prime | X_{k-n/2+1}, \ldots X_{k} \sim P_0,\; X_{k+1}, \ldots X_{k+n/2} \sim P_1] \le \frac{\beta}{8} \label{T_u}.
 \end{align}
 Condition \eqref{T_l} means that after taking a union bound over all the windows that do not contain $k^*$, the probability that $\AboveThreshold$ raises the alarm on the window that does not contain the true change point $k^*$ does not exceed $\beta/8$.  Condition \eqref{T_u} means that on the window containing the true change-point $k^*$ in the center of the window, $\AboveThreshold$ will fail to raise the alarm with probability at most $\beta/8$.


It will be helpful to have high probability bounds that the test statistics $U(k)$ are close to their means.  Using McDiarmid's Inequality (Theorem \ref{lem.mcd}) we can obtain that for any $k>n$
\begin{align}
\Pr[U(k)-\E[U(k)]>t]\le \exp(-t^2n/2),\label{mc_l}\\ 
\Pr[U(k)-\E[U(k)]<-t]\le \exp(-t^2n/2) \label{mc_u}
\end{align}

Using these bounds, we will first find $T_L^\prime$. Note that Condition \eqref{T_l} on $T_L'$ considers the setting where all points in the current window are drawn from $P_0$.  Under this condition, $\E[U(k)]=1/2$. Then by plugging in $t=T_L^\prime-1/2$ into Inequality \eqref{mc_l}, we get the following expression:
\begin{align*}
\Pr\left[U(k)\ge T_L^\prime | X_{k-n/2+1}, \ldots X_{k+n/2} \sim P_0\right]\le \exp\left(-\frac{n}{2}\left(T_L'-\frac{1}{2}\right)^2\right) 
\end{align*}
Setting the right hand side of this to less than or equal to $\frac{\beta}{8(k^* - n/2)}$ and solving for $T_L'$ gives the following lower bound, which satisfies Condition \eqref{T_l}:
\begin{align*}
T_L^\prime =\frac{1}{2}+ \sqrt{\frac{2}{n}\log\left( \frac{8(k^*-n/2)}{\beta} \right)}.
\end{align*}

Next we find the upper bound $T_U$. Note that Condition \eqref{T_u} on $T_U'$ considers the setting where the first $n/2$ points in the window are drawn from $P_0$ and the remaining $n/2$ points are drawn from $P_1$.  Under this condition, $\E[U(k)]=a$. Then plugging $t=a-T_U^\prime$ in Inequality \eqref{mc_u} and using Condition \eqref{T_u}, we get the following bound:
 \begin{align*}
     \Pr[U(k)\le T_U^\prime | X_, \ldots X_{n/2} \sim P_0,\; X_{n/2+1}, \ldots X_{n} \sim P_1] \leq \exp\left(-(a-T_U')^2n/2 \right) \le \frac{\beta}{8}. 
 \end{align*}
 Solving this for $T_U'$ gives the following Inequality which satisfies Condition \eqref{T_u}:
 \begin{equation*}
     T_U^\prime \leq a - \sqrt{\frac{2}{n}\log\left( \frac{8}{\beta} \right)}.
 \end{equation*}
 
We now return to account for the error from $\AboveThreshold$.  To ensure that this error does not cause a window to be called before the true change-point and also does not skip the window with the true change-point, we require the following conditions to both hold with probability $\frac{\beta}{4}$
 \begin{align*}
     &\text{For } T \ge T_L,\;\;U_k<T-\alpha^\prime \text{ when } k<k^*\\ 
     &\text{For } T \le T_U,\;\;U_{k^*}>T+\alpha^\prime 
 \end{align*}
 Thus we obtain that the new interval for $T$ is $[T_L,T_U]$, where $T_L=T_L^\prime+\alpha^\prime$, and $T_U=T_U^\prime-\alpha^\prime$.
 If both those conditions hold then for $\alpha^\prime=\frac{32\log(8(k^*-n/2)/\beta)}{ n\epsilon}$, $\AboveThreshold$ will identify the window which contains the true change point with probability $\left(1-\beta/4\right)$ by Theorem \ref{thm.atacc}. Taking a union bound over the failure probabilities of Conditions \eqref{T_l} and \eqref{T_u}, and the statement above, we can see that $\OnlinePNCPD$ will call $\PNCPD$ on the right window except with small probability $\beta/2$. 

Finally, we can use the accuracy guarantees of \PNCPD\ to show that conditioned on raising an alarm in the correct window, we are likely to output an estimate $\hat{k}$ that is close to the true change-point $k^*$. Slightly more careful accounting is needed here, because conditioning on raising an alarm and calling \PNCPD, the data points in the chosen window are no longer distributed according to the change-point model. 
Let $W(k)$ denote the event that $\OnlinePNCPD$ calls $\PNCPD(\set{x_{k-n/2+1+\gamma n}, \ldots x_{k+n/2+\gamma n}},\eps/2,\gamma)$ on the window centered at $k$. Then 
 \begin{align*}
     \Pr\left[\abs{\tilde{k}-k^*}>\alpha\right] &= \sum_{k>n/2}\Pr\left[W(k)\cap \{|\tilde{k}-k^*|>\alpha\}\right] \\
     &\le \sum_{k \notin (k^*-n/2,k^*]}\Pr\left[W(k)\right]+\sum_{k \in (k^*-n/2,k^*]}\Pr\left[W(k)\cap \left\{\abs{\tilde{k}-k^*}>\alpha\right\}\right] \\
     &\le \frac{\beta}{2} + \frac{n}{2} \Pr\left[\PNCPD \text{ fails}\right]<\beta
 \end{align*}
 To achieve the inequality above, the probability of $\PNCPD$ fails to report the change point within the $\alpha$-window around $k^*$ has to be bounded by $\beta/n$. Thus by Theorem \ref{private.acc} we set the error to be,
 \begin{equation*}
	\alpha=\max\left\{C_1\cdot \left(\frac{1}{\gamma^4 \left(a-1/2\right)^2}\right)^c\cdot \log \frac n \beta, C_2\cdot \left(\frac{1}{\epsilon\gamma \left(a-1/2\right)}\right)^c\cdot \log \frac n \beta \right\},
	\end{equation*}
	for any constant $c>1$ and some constant $C_1, C_2>0$ depending on $c$.
	 \end{proof}
	 
We have proved the theorem, but we should also show that the window $[T_L,T_U]$ is non-empty, and there exists a good range in which to choose the threshold $T$.   The condition that $T_L < T_U$ is equivalent to, 	
\begin{align}
   a -\frac{1}{2} > \sqrt{\frac{2}{n}\log\left( \frac{8(k^*-n/2)}{\beta} \right)}+\sqrt{\frac{2}{n}\log\left( \frac{8}{\beta} \right)}+\frac{64\log(8(k^*-n/2)/ \beta)}{ n\epsilon}.\label{inq1}
\end{align}
We can simplify Inequality \eqref{inq1} as,
\begin{align*}
  \sqrt{\frac{2}{n}\log\left( \frac{8(k^*-n/2)}{\beta} \right)}+\sqrt{\frac{2}{n}\log\left( \frac{8}{\beta} \right)}+\frac{64\log(8(k^*-n/2)/ \beta)}{ n\epsilon} \\ < 
  \sqrt{\frac{2}{n}\log\left( \frac{8k^*}{\beta} \right)}+\sqrt{\frac{2}{n}\log\left( \frac{8}{\beta} \right)}+\frac{64\log(8k^*/ \beta)}{ n\epsilon} <a - \frac 1 2.
\end{align*}
Finally, solving the right hand side for $n$, we find the following bound on $n$ that satisfies Inequality \eqref{inq1}.
\begin{align*}
  n> \frac{1}{(a-1/2)^2} \left(  \sqrt{2 \log{\left(\frac{8k^*}{\beta}\right)}} +\sqrt{2\log \left( \frac{8}{\beta}\right)} + \frac{64}{\epsilon} \log \left(\frac {8k^*} \beta\right)    \right)^2.
\end{align*}

For any starting database size that is at least this large (only $n=\Omega((\frac{\log (k^*/\beta)}{\eps(a-1/2)})^2)$
), the acceptable region $[T_L,T_U]$ for a threshold $T$ will be non-empty. Moreover, the $\log k^*$ dependence of $T_L$ and $T_U$ means that only a rough estimate of the true change-point is necessary in practice to choose an acceptable threshold $T$.
 

%% file: app.tex
\section{Application: Drift Change Detection}\label{s.drift}
In this section, we extend our consideration of the change-point problem to the setting where data are not
sampled i.i.d. from fixed pre- and post-change distributions, but instead are sampled from distributions that
are changing smoothly over time. In particular, we consider distributions with \emph{drift}, where the parameter of
the distribution changes linearly with time, and the rate of linear drift changes at the change-point.  Since the samples are not i.i.d., we consider differences between successive pairs of samples in order to apply the algorithms from the previous sections. 


The {\em drift change detection problem} is parametrized by error terms $e_t$ independently sampled from a mean-zero distribution $\mathcal{S}$, two drift terms $\xi_0$ and $\xi_1$, a drift change-point $t^*\in[n]$, and a mean $\eta$ associated with $t^*$. Independent random variables $X=\set{x_1,\dots,x_n}$ are said to be drawn from the drift change detection model if we can write 
\begin{equation*}
    x_t = \mu_t + e_t, 
\end{equation*}
for $\mu_t$ piecewise linear as follows:
\begin{equation*}
\mu_t = \begin{cases}
\eta - (t^*-t)\xi_0 & t\le t^* \\
\eta + (t-t^*)\xi_1 & t> t^*
\end{cases}. 
\end{equation*}
%
Our goal is to detect the drift change-point $t^*$ with the smallest possible error. 

In order to apply our algorithms which require i.i.d.~samples, we will transform the sample $X$ by considering differences of consecutive pairs of $x_i$.  These differences are i.i.d.~with mean $\xi_0$ before $t^*$, and i.i.d.~with mean $\xi_1$ after $t^*$, and we can now apply \PNCPD\ to this instance of change-point detection. For ease of presentation, we will assume $n$ is even and $t^*$ is odd.

Formally, define a new sample $Y = \{y_1, \ldots, y_{n/2}\}$ with sample points $y_t = x_{2t}-x_{2t-1}$, for $t=1,\ldots n/2$. Then we have 
\begin{equation*}
    y_t =
    \begin{cases}
     \xi_0 + e_{2t}-e_{2t-1}, & \text{for } t=1,\ldots, \frac{t^*-1}{2} , \\
     \xi_1 + e_{2t}-e_{2t-1}, & \text{for } t=\frac{t^*+2}{2}, \ldots, \frac{N}{2}.
    \end{cases}
\end{equation*}
Note that random variables $(e_{2t}-e_{2t-1})$ are independent and identically distributed. 
Thus the $y_t$ are independent, and they are sampled from a fixed distribution before the change point, and from another distribution after the change-point. We can then apply the \PNCPD\ algorithm and privately estimate the drift change-point $\hat{t}$ as twice the output of \PNCPD($\set{y_1,\dots,y_{n/2}},\eps,\gamma$).  This estimation procedure will inherit the privacy and accuracy results of Theorems \ref{thm.offpriv} and \ref{private.acc}.\footnote{This procedure finds a change-point in the sample $Y$, which corresponds to a pair $(x_{2t-1}, x_{2t})$ such that one of them is the estimated change point. Under our assumption that $t^*$ is odd, we should output $\hat{t} = 2t-1$. If $t^*$ is even, then the estimated change-point may be off by one, and $y_{t^*/2}$ is distributed differently than other data points. However, since the $\PNCPD$ algorithm is differentially private, its performance is guaranteed be in insensitive to a single outlier in the database, so this fact will not affect the result of the algorithm by too much.}

As a concrete example, consider points sampled from a Gaussian distribution with mean $\mu_t = \xi_0 t +\eta_0$ and standard deviation $\sigma$ for $t \leq t^*$, and from a Gaussian distribution with mean $\mu_t = \xi_1 t +\eta_1$ and standard deviation $\sigma$ for $t > t^*$.  Then $y_t=x_{2t}-x_{2t-1}$ will be Gaussian with variance $2\sigma^2$ and mean $\xi_0$ before the change-point and $\xi_1$ after it.  If any of the parameters $\xi_0, \xi_1$, or $\sigma$ are unknown, this would require nonparametric change-point estimation.

\begin{corollary} \label{smooth.theorem}
For data $X=\set{x_1,\dots,x_n}$ drawn according to the drift change model with  drift terms $\xi_0> \xi_1$,
 constraint $\gamma\in(0,1/2)$, drift change time $t^*\in(\ceil{\frac{\gamma}{2} n}\ldots \ceil{(1-\frac{\gamma}{2})n})$, and privacy parameter $\eps>0$,  
  there exists an $\epsilon$-differentially private nonparametric change point estimator that is $(\alpha, \beta)$-accurate for any $\beta > 0$ and 
	\begin{equation*}
	\alpha=\max\left\{C_1\cdot \left(\frac{1}{\gamma^4 (a-1/2)^2}\right)^c\cdot \log \frac 1 \beta, C_2\cdot \left(\frac{1}{\epsilon\gamma (a-1/2)}\right)^c\cdot \log \frac 1 \beta \right\}, 
	\end{equation*}
	for any constant $c>1$ and some constant $C_1, C_2>0$ depending on $c$.
\end{corollary}

We note that this approach is not restricted solely to offline linear drift detection.  The same reduction in the online setting would allow us to use \OnlinePNCPD\ to detect drift changes online. Additionally, a similar approach could be used to for other types of smoothly changing data, as long as the smooth changes exhibited enough structure to allow for reduction to the i.i.d.~setting.  For example, if data were sampled of the form  $x_t = f(\mu_t + e_t)$ for any one-to-one function $f:\R\to\R$, we could define $y_t = f^{-1}(x_{2t})-f^{-1}(x_{2t-1})$, and these $y_t$s would again be i.i.d..  This includes random variables of the form $\exp(\mu_t+e_t)$, $\log(\mu_t+e_t)$, and arbitrary polynomials $(\mu_t+e_t)^k$ (where even-degree polynomials must be restricted to, e.g., only have positive range).

%% file: simulations.tex
\section{Empirical Results}\label{s.sim}

We now report the results of an experiment on real data followed by Monte Carlo experiments designed to validate the theoretical results of previous sections.  We only consider our accuracy guarantees because the nature of differential privacy provides a strong worst-case guarantee for all hypothetical databases, and therefore is impractical and redundant to test empirically. Our simulations consider both offline and online settings for detecting a change in the mean of Gaussian distribution.


\subsection{Results of Offline Algorithm with Real Data}

First we illustrate the effectiveness of our offline algorithm on real data by applying it to a window of stock price data including a sudden drop in price, and we use it to determine approximately when this change-point occurred. We use a dataset from \cite{cao2018multi}, which contains stock price data over time, with prices collected every second over a span of 5 hours on October 9, 2012. We identified by visual inspection a window of $n=200$ seconds (indexed 6900 to 7100 in the dataset, reindexed 0 to 200 here) that appears to include a discrete change in distribution from higher mean price to lower mean price. 
We then calculated the argmax of the Mann-Whitney statistic $V(k)$ to identify the most likely change-point as time $\hat k=92$, assuming the pre-change data were drawn i.i.d. from one distribution and the post-change data were drawn i.i.d. from a distribution with lower mean. We used this estimate as the ground truth ($k^*=\hat k=92$) in error analysis of our private offline algorithm. We ran our \PNCPD\ algorithm with $\gamma=0.1$ on the selected dataset $10^3$ times for each privacy value $\eps=0.1,0.5,1$. Figure~\ref{fig:realdata}(a) plots the data in our selected window, and Figure~\ref{fig:realdata}(b) plots the empirical accuracy $\beta = \Pr[|\tilde{k}-k^\ast|>\alpha]$ as a function of $\alpha$ for our \PNCPD\ simulations.

\begin{center}
\begin{minipage}{.9\linewidth}
\begin{figure}[H]
	\centering
	\subfloat[][Data trajectory]{\includegraphics[width=.45\textwidth]{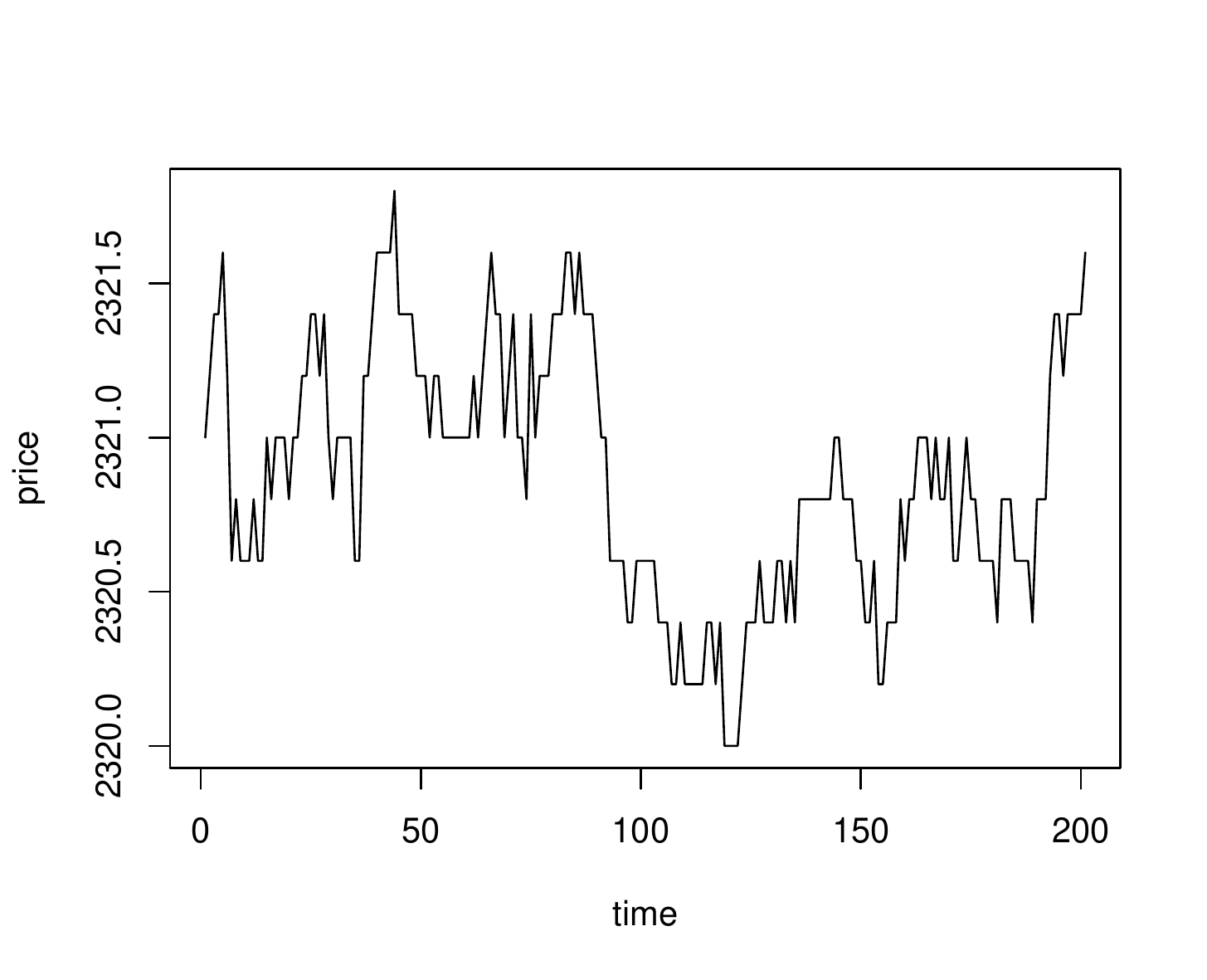}}
	\subfloat[][Accuracy of \PNCPD\ on data from (a)]{\includegraphics[width=.45\textwidth]{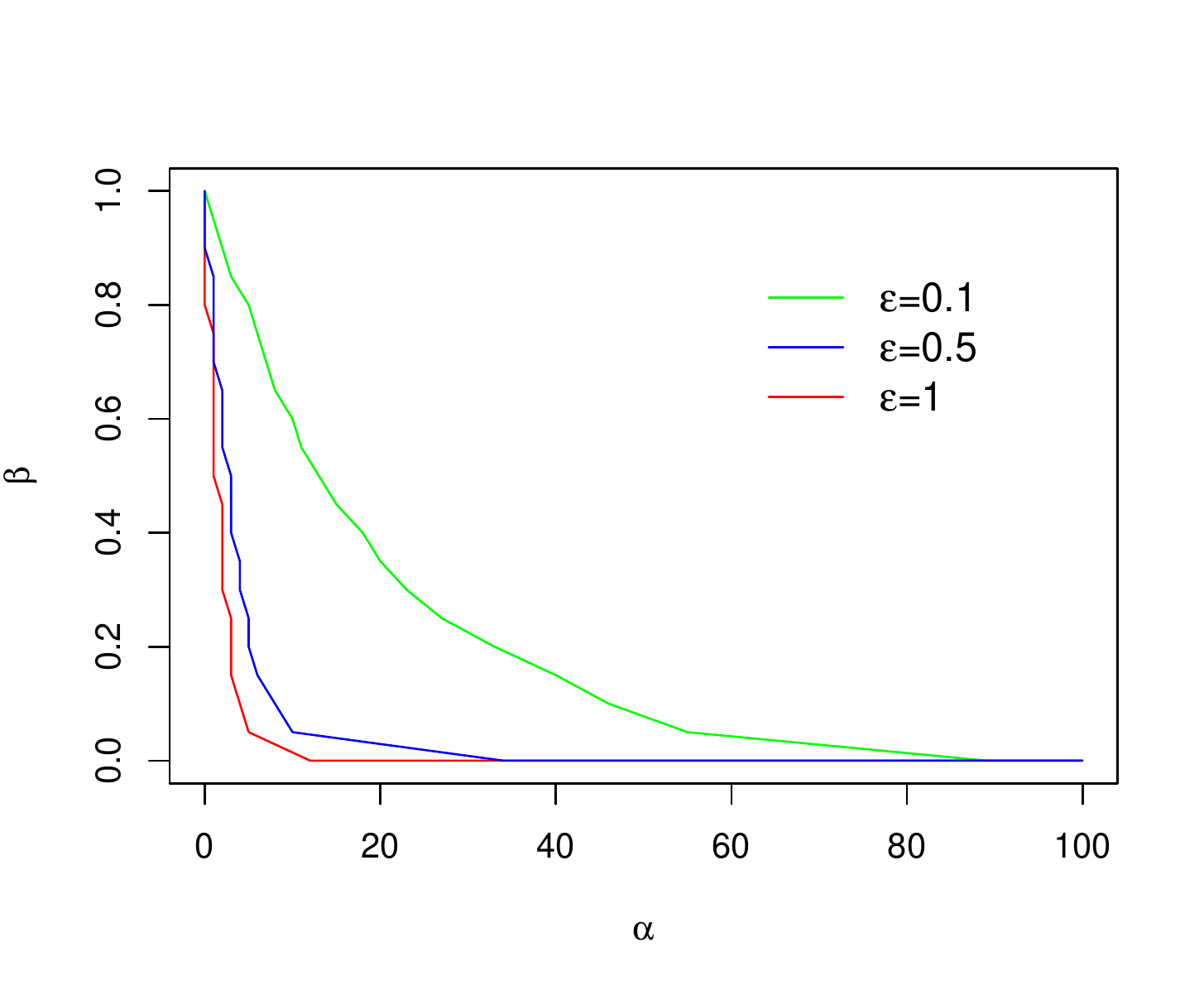}}
	\caption{\small Real data and accuracy results. 
	} 
	\label{fig:realdata}
\end{figure}
\end{minipage}
\end{center}


\subsection{Offline Results with Synthetic Data}

We now provide simulations of our algorithms using many synthetic datasets drawn exactly according to the change-point model. In the following simulations for \PNCPD, we use an initial distribution of $\mathcal N(0,1)$ and post-change distributions of the form $\mathcal N(\mu_1,1)$, considering both a small change $\mu_1=1$ and a large change $\mu_1=5$. We use $n=200$ observations where the true change occurs at time points $k^{*} = 50, 100, 150$. This process is repeated $10^3$ times for each value of $k^*$ and $\mu_1$. We consider the performance of our algorithm for $\gamma=0.1$ and $\eps=0.1,1,5,\infty$, where $\eps=\infty$ corresponds to the non-private problem, which serves as our baseline.  The results are summarized in Figure \ref{fig:offline}, which plots  the empirical probabilities $\beta = \Pr[|\tilde{k}-k^\ast|>\alpha]$ as a function of $\alpha$.


\begin{center}
\begin{minipage}{.9\linewidth}
\begin{figure}[H]
	\centering
	\subfloat[][$k^*=50$, $\mu_1=5$]{\includegraphics[width=.33\textwidth]{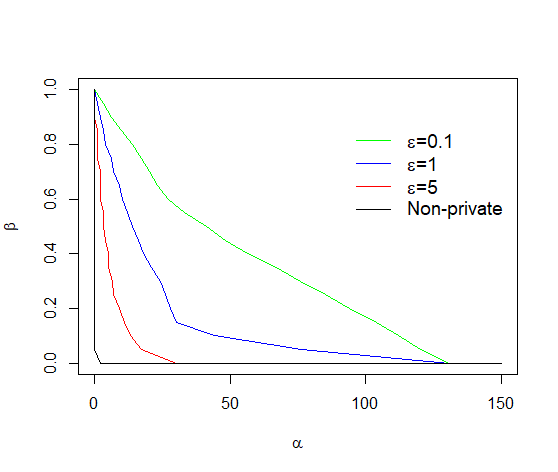}}
	\subfloat[][$k^*=100$, $\mu_1=5$]{\includegraphics[width=.33\textwidth]{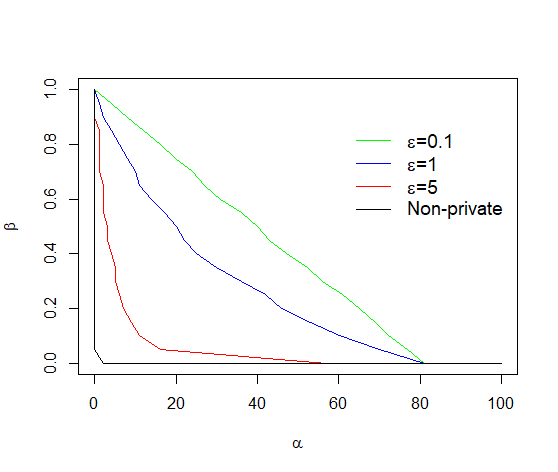}}
	\subfloat[][$k^*=150$, $\mu_1=5$]{\includegraphics[width=.33\textwidth]{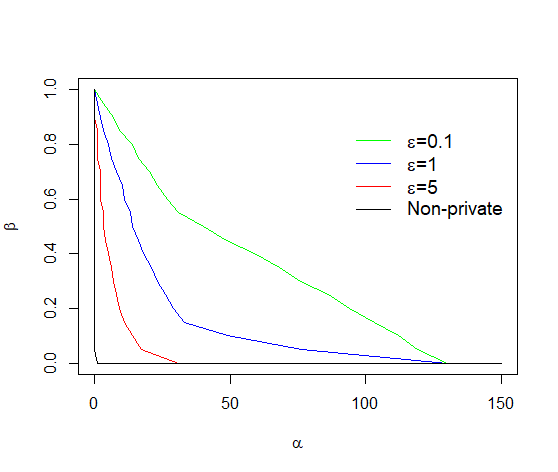}}\\
	\subfloat[][$k^*=50$, $\mu_1=1$]{\includegraphics[width=.33\textwidth]{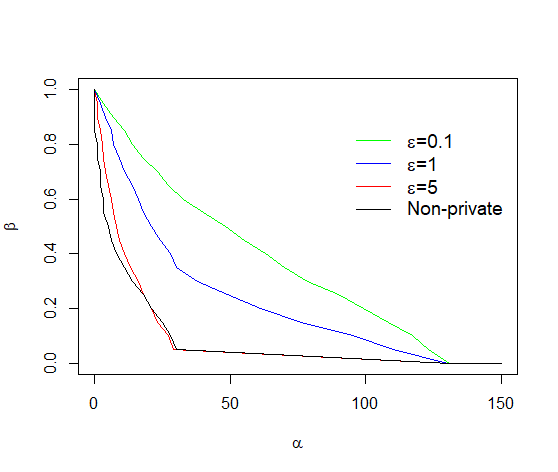}}
	\subfloat[][$k^*=100$, $\mu_1=1$]{\includegraphics[width=.33\textwidth]{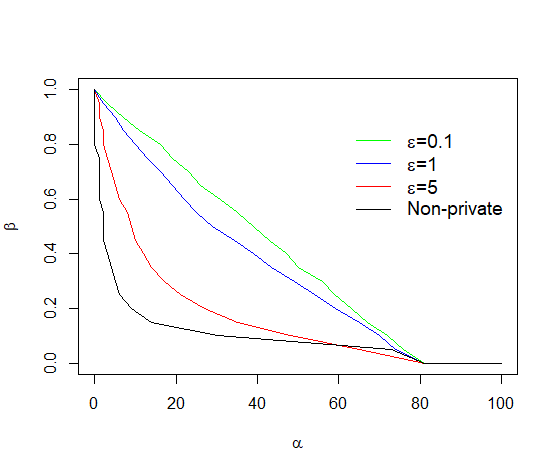}}
	\subfloat[][$k^*=150$, $\mu_1=1$]{\includegraphics[width=.33\textwidth]{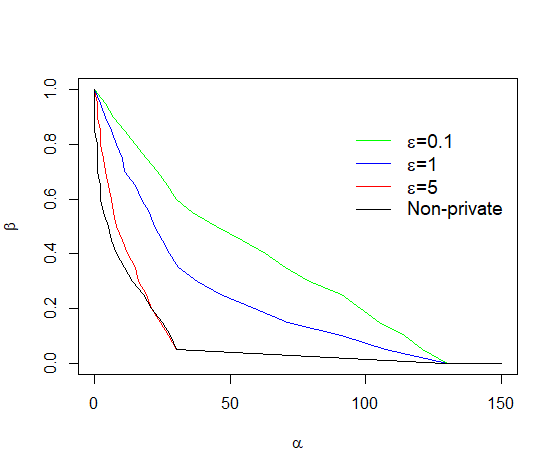}}
	\caption{\small Empirical accuracy $\beta = \Pr[|\tilde{k}-k^\ast|>\alpha]$ of \PNCPD~from Monte Carlo simulations using Gaussian data, where pre-change data are drawn from $\mathcal N(0,1)$ and post-change data are drawn from $\mathcal N(\mu_1,1)$.  Each simulation involves $10^3$ runs of \PNCPD~with varying $\eps$ on data generated by 200 i.i.d. samples from appropriate distributions with $\mu_1=1$ or $5$, and change point $k^*=50,100$, or $150$.
%
	} \label{fig:offline}
\end{figure}
\end{minipage}
\end{center}

As expected, the algorithm finds the change-point accurately, with better performance when the distributional change is larger or the $\epsilon$ value is larger.  Performance is slightly diminished when the change-point is at the center of the window, corresponding to $k^*=100$ in our experiments.  This is due to the scaling factor $\frac{1}{k(n-k)}$ in the expression of $V(k)$ as seen in Equation \eqref{eq.Vk}, which places relatively higher weight on $k$ that are close to the beginning and end of the window.  This scaling factor could be removed and our algorithm would still be differentially private and our accuracy result would (qualitatively) continue to hold for change-points near the center of the window. However, if an analyst already has reason to believe that the change-point occurs in the middle of her selected window, she is unlikely to need a change-point detection algorithm.

We also note that our simulations use slightly larger $\epsilon$ values and distributional changes than previous work on \emph{parametric} private change-point detection, where the pre- and post-change distributions are given explicitly as input to the algorithm \cite{CKM+18}.\footnote{The simulations of \cite{CKM+18} used $\epsilon=0.1, 0.5, 1$ and $\mu_1=0.5,1$, whereas we use $\epsilon=0.1, 1, 5$ and $\mu_1=1,5$.} This is to be expected since the nonparametric problem is information theoretically harder to solve, because the test statistic cannot be tailored to the pre- and post-change distributions.

To illustrate these accuracy guarantees, Figure \ref{fig:statistic} show the values of the true test statistic $V(k)$ and the noisy test statistic $V(k) + Z_k$ for the same distributions.  We still use $n=200$ observations and $k^*=50,100$ and $\mu_1=1,5$, and run the process only once for each pair of parameter values.  We note that for the chosen distributions, $a<1/2$ so our test statistic $V(k)$ should be minimized at $k^*$, and we use the variant of \PNCPD\ that outputs $\tilde{k}=\argmin \{V(k)+Z_k\}$ rather than the argmin as described in Algorithm \ref{algo.offline}. The smoother black line in the figures corresponds to the true test statistic $V(k)$ and the more jagged orange line corresponds to the noisy test statistic $V(k) + Z_k$ for $\epsilon=5$.  Figure \ref{fig:statistic} shows that in all cases, the true statistic is minimized at the true change $k^*$.  This is even more prominent when the distributional change is larger ($\mu_1=5$), so more noise can be tolerated. Under smaller distributional changes ($\mu_1=1$) the minimization of $V(k)$ around $k^*$ is less dramatic, and there is more opportunity for the noise terms $Z_k$ to introduce estimation error when minimizing the noisy statistic.  This also illustrates the structure of the proof of Theorem \ref{private.acc}, and in particular Equation \eqref{bad.each}, where we separate out the failure probability of the algorithm into two terms: the probability of bad data and the probability of bad draws from the Laplace distribution.  


\begin{center}
\begin{minipage}{.9\linewidth}
\begin{figure}[H]
	\centering
	\subfloat[][$k^*=50$, $\mu_1=5$]{\includegraphics[width=.33\textwidth]{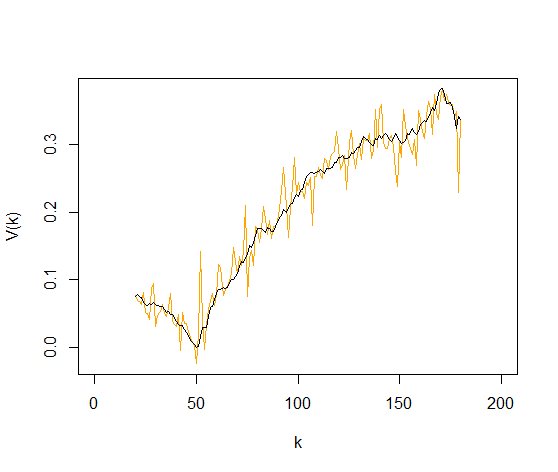}}
	\subfloat[][$k^*=100$, $\mu_1=5$]{\includegraphics[width=.33\textwidth]{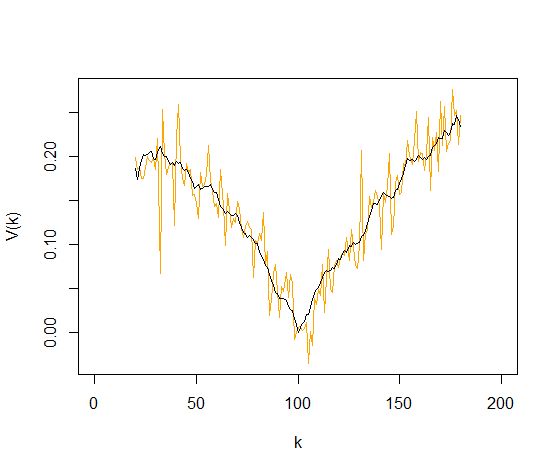}}
	\\
	\subfloat[][$k^*=50$, $\mu_1=1$]{\includegraphics[width=.33\textwidth]{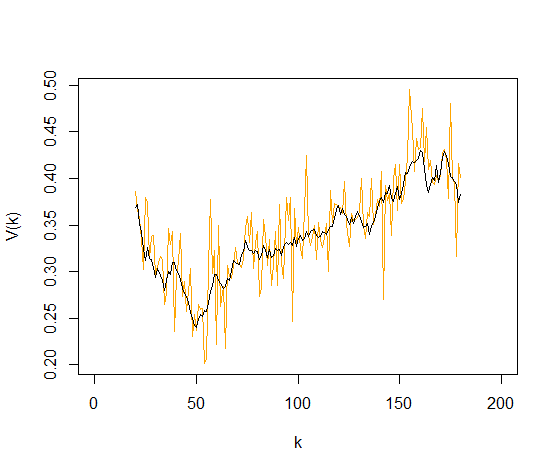}}
	\subfloat[][$k^*=100$, $\mu_1=1$]{\includegraphics[width=.33\textwidth]{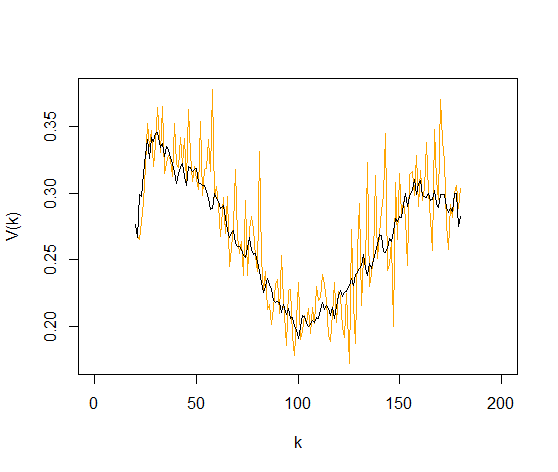}}
	\caption{\small Value for statistics $V(k)$ with (orange) and without (black) Laplace noise with privacy parameter $\epsilon=5$ for varying settings for the size change and location of a change point.
	} \label{fig:statistic}
\end{figure}
\end{minipage}
\end{center}

Finally, we provide simulations of our \PNCPD~algorithm for our application of drift change detection, as described in Section \ref{s.drift}. Recall that our drift change detection model involved data points $X=\{x_1, \ldots, x_n\}$ defined as $x_t=\mu_t + e_t$ where 
\begin{equation*}
\mu_t = \begin{cases}
\eta - (t^*-t)\xi_0 & t\le t^* \\
\eta + (t-t^*)\xi_1 & t> t^*
\end{cases}, 
\end{equation*}
for drift change-point $t^*$, and $e_t$ are mean-zero noise terms. In our simulation we use parameters $\eta=1$, $\xi_0=0$, $\xi_1=5$, and $e_t \sim_{i.i.d.} \mathcal{N}(0,1)$. We use $n=200$ observations where the true drift change occurs at time $t^{*} =100$, and repeat the process $10^3$ times. We modify the observations $X$ to create a new sample $Y=\{y_1,\ldots,y_{n/2}\}$ as described in Section \ref{s.drift}, and apply our \PNCPD~algorithm to this new sample.  Figure \ref{fig:smooth} plots the empirical accuracy $\beta = \Pr[|\tilde{t}-k^\ast|>\alpha]$ as a function of $\alpha$ for $\gamma=0.1$ and $\eps=0.1,1,5,\infty$, where $\eps=\infty$ is our non-private baseline.


\begin{center}
\begin{minipage}{.9\linewidth}
\begin{figure}[H]
	\centering
	\includegraphics[width=.33\textwidth]{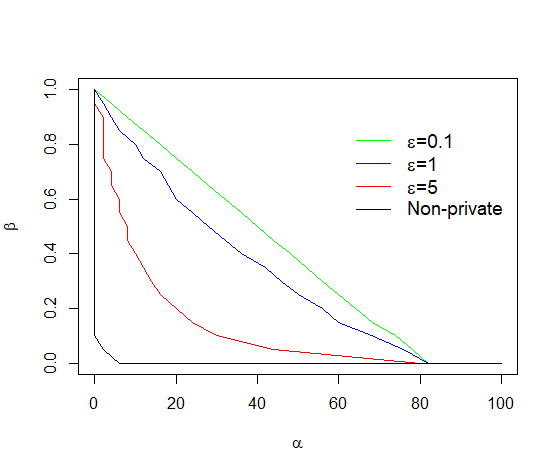}
	\caption{\small Empirical accuracy $\beta = \Pr[|\tilde{t}-t^*|>\alpha]$ of \PNCPD~for drift detection. The data are generated from the drift change model with parameters $\eta=1$, $\xi_0=0$, $\xi_1=5$, and $e_t$ drawn from $\mathcal{N}(0,1)$. These data are then modified as described in Section \ref{s.drift} so that the \PNCPD~algorithm can be applied.
	} \label{fig:smooth}
\end{figure}
\end{minipage}
\end{center}

\subsection{Online Results with Synthetic Data}
We also perform simulations for our online private change-point detection algorithm \OnlinePNCPD\, when the data points arrive sequentially. We use an initial distribution of $\mathcal N(5,1)$ and post-change distribution of $\mathcal N(0,1)$, where the true change occurs at time $k^{*} = 5000$.  To help ensure that the range of the appropriate threshold $T$ in \OnlinePNCPD\ is non-empty, we choose a larger window size $n=500$, and larger privacy parameter $\epsilon=1,5, 10, \infty$. 

We choose the appropriate threshold $T$ by setting a constraint that an algorithm must have positive and negative false alarm rates both at most $0.1$, which can be ensured by setting $\beta=0.4$. (Recall from the proof of Theorem \ref{online.acc} that our false alarm rates are each $\beta/4$.) 
Since we know $k^*$ and $a$, we can compute the theoretical upper and lower bounds on the threshold exactly for the distributions used in our simulations using the expressions given in the statement of Theorem \ref{online.acc}. The resulting lower bounds are $T_L=1.28,0.80,0.74,0.69$ and the upper bounds are $T_U=0.16,0.74,0.81,0.89$ for $\eps=1,5, 10, \infty$, respectively. Although the theoretical range of $T$ is empty for $\eps=1,5$, our empirical results show that $T=0.8$ is sufficient to control both false alarm rates, as the theoretical bounds are overly conservative. We choose $T=0.8$ for all $\eps=1,5,10,\infty$. In practice when $a$ and $k^*$ are unknown, the analyst should set $a$ to be the smallest interesting magnitude of distributional change, and $k^*$ to be the analyst's estimate of the time of the change, and similarly compute $T_L$ and $T_U$ using these estimates. Larger values of $k^*$ correspond to more conservative estimates and result in smaller windows for the threshold.  We also note the analyst can also choose the lower and upper bounds of $T$ via numerical methods as in \cite{CKM+18}. 

We run our \OnlinePNCPD\ algorithm $10^3$ times with $\gamma=0.1$ and privacy parameters $\eps=1,5, 10, \infty$. Figure \ref{fig:online} summarizes these simulation results. As in the proof of Theorem \ref{online.acc}, we can separate the error into two possible sources within the algorithm: halting on an incorrect window, and producing an incorrect estimate of the change-point, even conditioned on halting on the correct window.  Figure \ref{fig:online}(a) shows the error from both of these sources, and Figure \ref{fig:online}(b) shows the error from only the latter source.  These figures show that our algorithm works well with privacy parameters $\eps=5,10,\infty$. For $\eps=1$, we can control the overall error rate to be less than $0.4$ as desired, but not much lower.  Figure \ref{fig:online}(b) shows that this error mainly comes from the failure to halt on the window that contains the true change-point, because the error decreases dramatically after conditioning on the algorithm halting on a correct window that contains the true change-point.

\begin{center}
	\begin{minipage}{.9\linewidth}
		\begin{figure}[H]
			\centering
			\subfloat[][Error probability from inaccurate estimate \\and false alarm]{\includegraphics[width=.45\textwidth]{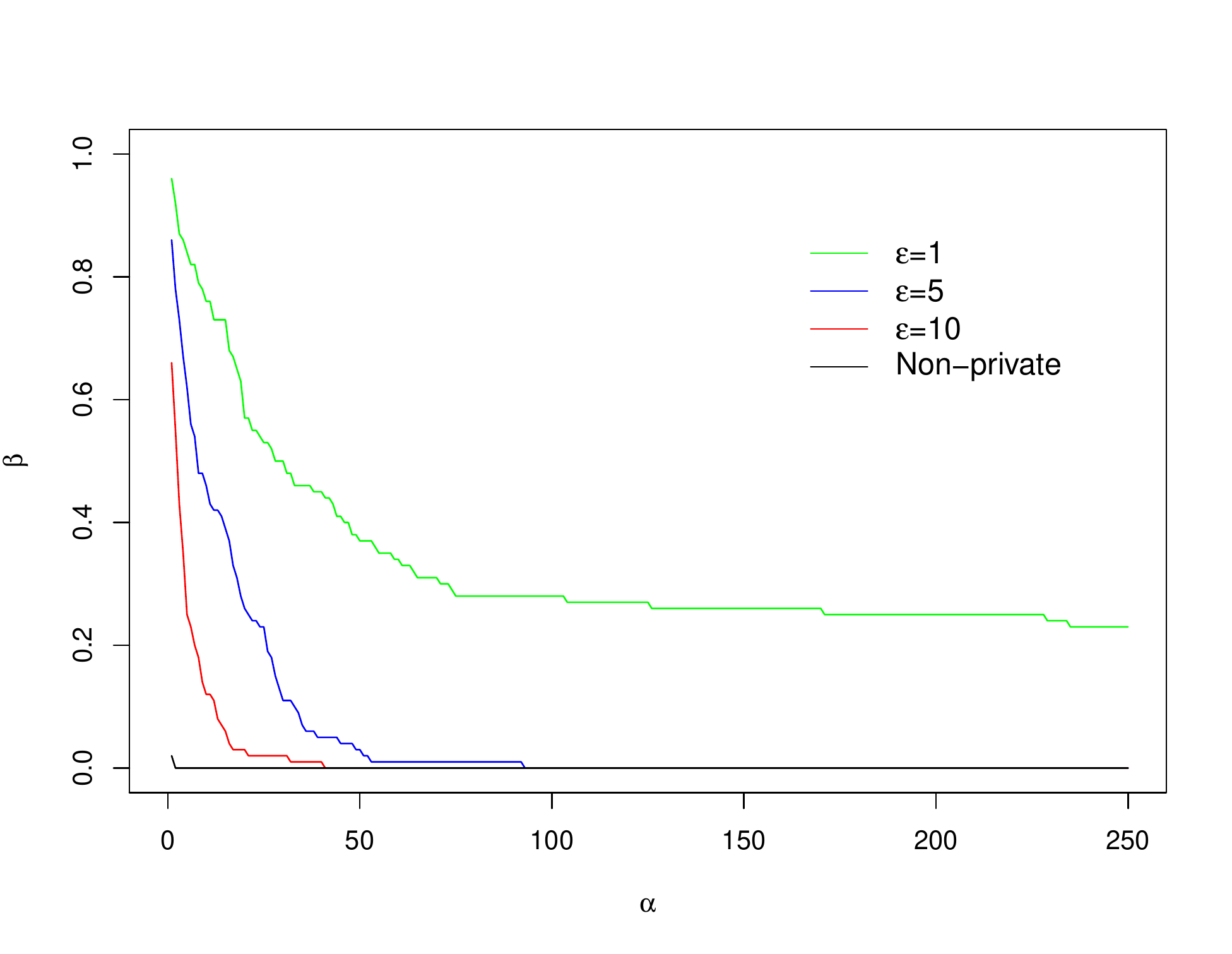}}
			\subfloat[][Error probability from inaccurate estimate only]{\includegraphics[width=.45\textwidth]{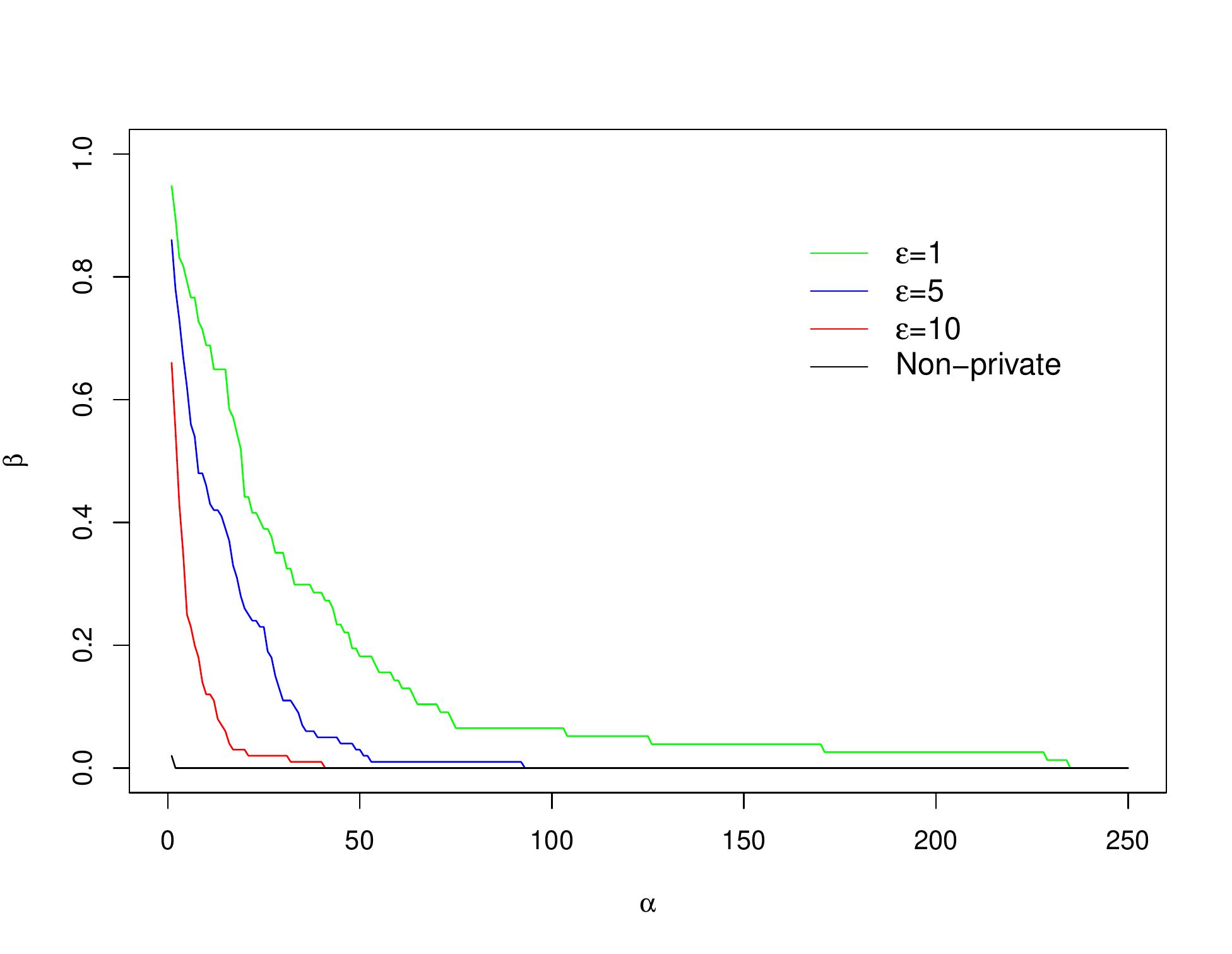}}
		 \caption{\small Probability of inaccurate estimation and false alarm (left) and probability of inaccurate report conditioned on raising an alarm correctly (right) for Monte Carlo simulations. Data drawn from $\mathcal{N}(5,1)$ pre-change and $\mathcal{N}(0,1)$ post-change, with true change-point $k^*=5000$. Each simulation involves $10^3$ runs of \OnlinePNCPD\ with $\gamma=0.1$, window size $n=500$, threshold $T=0.8$, and varying $\eps$. }\label{fig:online}
		\end{figure}
	\end{minipage}
\end{center}


%% file: main.bbl
\newcommand{\etalchar}[1]{$^{#1}$}
\begin{thebibliography}{CKM{\etalchar{+}}18b}

\bibitem[ABH89]{azzalini1989use}
Adelchi Azzalini, Adrian~W Bowman, and Wolfgang H{\"a}rdle.
\newblock On the use of nonparametric regression for model checking.
\newblock {\em Biometrika}, 76(1):1--11, 1989.

\bibitem[BJ68]{bhattacharyya1968nonparametric}
GK~Bhattacharyya and Richard~A Johnson.
\newblock Nonparametric tests for shift at an unknown time point.
\newblock {\em The Annals of Mathematical Statistics}, 39(5):1731--1743, 1968.

\bibitem[BP03]{Bai:Perron:2003}
J.~Bai and P.~Perron.
\newblock Computation and analysis of multiple structural change models.
\newblock {\em Journal of Applied Econometrics}, 18(1):1--22, 2003.

\bibitem[Car88]{carlstein1988nonparametric}
Edward Carlstein.
\newblock Nonparametric change-point estimation.
\newblock {\em The Annals of Statistics}, 16(1):188--197, 1988.

\bibitem[Cha17]{chan:2017}
H.~P. Chan.
\newblock Optimal sequential detection in multi-stream data.
\newblock {\em The Annals of Statistics}, 45(6):2736--2763, 2017.

\bibitem[CKM{\etalchar{+}}18a]{canonne2018structure}
Cl{\'e}ment~L Canonne, Gautam Kamath, Audra McMillan, Adam Smith, and Jonathan
  Ullman.
\newblock The structure of optimal private tests for simple hypotheses.
\newblock {\em arXiv preprint arXiv:1811.11148}, 2018.

\bibitem[CKM{\etalchar{+}}18b]{CKM+18}
Rachel Cummings, Sara Krehbiel, Yajun Mei, Rui Tuo, and Wanrong Zhang.
\newblock Differentially private change-point detection.
\newblock In {\em Proceedings of the 32nd International Conference on Neural
  Information Processing Systems}, NeurIPS '18, pages 10848--10857, 2018.

\bibitem[CKS{\etalchar{+}}18]{CKS+18}
Simon Couch, Zeki Kazan, Kaiyan Shi, Andrew Bray, and Adam Groce.
\newblock A differentially private wilcoxon signed-rank test.
\newblock arXiv pre-print 1809.01635, 2018.

\bibitem[CKS{\etalchar{+}}19]{CKS+19}
Simon Couch, Zeki Kazan, Kaiyan Shi, Andrew Bray, and Adam Groce.
\newblock Differentially private nonparametric hypothesis testing.
\newblock arXiv pre-print 1903.09364, 2019.

\bibitem[CXG18]{cao2018multi}
Yang Cao, Yao Xie, and Nagi Gebraeel.
\newblock Multi-sensor slope change detection.
\newblock {\em Annals of Operations Research}, 263(1-2):163--189, 2018.

\bibitem[Dar76]{darkhovskh1976nonparametric}
BS~Darkhovskh.
\newblock A nonparametric method for the a posteriori detection of the
  ``disorder'' time of a sequence of independent random variables.
\newblock {\em Theory of Probability \& Its Applications}, 21(1):178--183,
  1976.

\bibitem[DMNS06]{DMNS06}
Cynthia Dwork, Frank McSherry, Kobbi Nissim, and Adam Smith.
\newblock Calibrating noise to sensitivity in private data analysis.
\newblock In {\em Proceedings of the 3rd Conference on Theory of Cryptography},
  TCC '06, pages 265--284, 2006.

\bibitem[DNPR10]{DNPR10}
Cynthia Dwork, Moni Naor, Toniann Pitassi, and Guy~N. Rothblum.
\newblock Differential privacy under continual observation.
\newblock In {\em 42nd ACM Symposium on Theory of Computing}, STOC '10, 2010.

\bibitem[DNR{\etalchar{+}}09]{DNRRV09}
Cynthia Dwork, Moni Naor, Omer Reingold, Guy~N. Rothblum, and Salil~P. Vadhan.
\newblock On the complexity of differentially private data release: efficient
  algorithms and hardness results.
\newblock In {\em Proceedings of the 41st ACM Symposium on Theory of
  Computing}, STOC '09, pages 381--390, 2009.

\bibitem[DR14]{dwork2014algorithmic}
Cynthia Dwork and Aaron Roth.
\newblock The algorithmic foundations of differential privacy.
\newblock {\em Foundations and Trends in Theoretical Computer Science},
  9(3--4):211--407, 2014.

\bibitem[GC11]{gibbons2011nonparametric}
Jean~Dickinson Gibbons and Subhabrata Chakraborti.
\newblock {\em Nonparametric statistical inference}.
\newblock Springer, 2011.

\bibitem[Kul01]{kulldorff:2001}
M.~Kulldorff.
\newblock Prospective time periodic geographical disease surveillance using a
  scan statistic.
\newblock {\em Journal of the Royal Statistical Society, Series A},
  164(1):61--72, 2001.

\bibitem[Lai95]{lai:1995}
T.~L. Lai.
\newblock Sequential changepoint detection in quality control and dynamical
  systems.
\newblock {\em Journal of the Royal Statistical Society, Series B},
  57(4):613--658, 1995.

\bibitem[Lai01]{lai:2001}
T.~L. Lai.
\newblock Sequential analysis: {s}ome classical problems and new challenges.
\newblock {\em Statistica Sinica}, 11(2):303--408, 2001.

\bibitem[Lil67]{lilliefors1967kolmogorov}
Hubert~W Lilliefors.
\newblock On the kolmogorov-smirnov test for normality with mean and variance
  unknown.
\newblock {\em Journal of the American statistical Association},
  62(318):399--402, 1967.

\bibitem[Lor71]{lorden:1971}
G.~Lorden.
\newblock Procedures for reacting to a change in distribution.
\newblock {\em The Annals of Mathematical Statistics}, 42(6):1897--1908, 1971.

\bibitem[LR02]{Lund:2002}
R.~Lund and J.~Reeves.
\newblock Detection of undocumented changepoints: A revision of the two-phase
  regression model.
\newblock {\em Journal of Climate}, 15(17):2547--2554, 2002.

\bibitem[McD89]{McD89}
Colin McDiarmid.
\newblock On the method of bounded differences.
\newblock In {\em Surveys in Combinatorics}, pages 148--188. Cambridger
  University Press, 1989.

\bibitem[Mei06]{mei:2006a}
Y.~Mei.
\newblock Sequential change-point detection when unknown parameters are present
  in the pre-change distribution.
\newblock {\em The Annals of Statistics}, 34(1):92--122, 2006.

\bibitem[Mei08]{mei:2008a}
Y.~Mei.
\newblock Is average run length to false alarm always an informative criterion?
\newblock {\em Sequential Analysis}, 27(4):354--419, 2008.

\bibitem[Mei10]{mei:2010}
Y.~Mei.
\newblock Efficient scalable schemes for monitoring a large number of data
  streams.
\newblock {\em Biometrika}, 97(2):419--433, 2010.

\bibitem[Mou86]{moustakides:1986}
G.~V. Moustakides.
\newblock Optimal stopping times for detecting changes in distributions.
\newblock {\em The Annals of Statistics}, 14(4):1379--1387, 1986.

\bibitem[MW47]{mann1947test}
Henry~B Mann and Donald~R Whitney.
\newblock On a test of whether one of two random variables is stochastically
  larger than the other.
\newblock {\em The annals of mathematical statistics}, pages 50--60, 1947.

\bibitem[Pag54]{page:1954}
E.~S. Page.
\newblock Continuous inspection schemes.
\newblock {\em Biometrika}, 41(1/2):100--115, 1954.

\bibitem[Par62]{parzen1962estimation}
Emanuel Parzen.
\newblock On estimation of a probability density function and mode.
\newblock {\em The annals of mathematical statistics}, 33(3):1065--1076, 1962.

\bibitem[Pol85]{pollak:1985}
M.~Pollak.
\newblock Optimal detection of a change in distribution.
\newblock {\em The Annals of Statistics}, 13(1):206--227, 1985.

\bibitem[Pol87]{pollak:1987}
M.~Pollak.
\newblock Average run lengths of an optimal method of detecting a change in
  distribution.
\newblock {\em The Annals of Statistics}, 15(2):749--779, 1987.

\bibitem[Rob66]{roberts:1966}
S.~W. Roberts.
\newblock A comparison of some control chart procedures.
\newblock {\em Technometrics}, 8(3):411--430, 1966.

\bibitem[Ros56]{rosenblatt1956remarks}
Murray Rosenblatt.
\newblock Remarks on some nonparametric estimates of a density function.
\newblock {\em The Annals of Mathematical Statistics}, pages 832--837, 1956.

\bibitem[She31]{shewhart:1931}
W.~A. Shewhart.
\newblock {\em Economic Control of Quality of Manufactured Product}.
\newblock D. Van Norstrand Company, Inc., 1931.

\bibitem[Shi63]{shiryaev:1963}
A.~N. Shiryaev.
\newblock On optimum methods in quickest detection problems.
\newblock {\em Theory of Probability \& Its Applications}, 8(1):22--46, 1963.

\bibitem[Wil45]{wilcoxon1945individual}
Frank Wilcoxon.
\newblock Individual comparisons by ranking methods.
\newblock {\em Biometrics bulletin}, 1(6):80--83, 1945.

\bibitem[WW40]{wald1940test}
Abraham Wald and Jacob Wolfowitz.
\newblock On a test whether two samples are from the same population.
\newblock {\em The Annals of Mathematical Statistics}, 11(2):147--162, 1940.

\bibitem[ZS12]{Zhang:Siegmund:2012}
N.~Zhang and D.~O. Siegmund.
\newblock Model selection for high-dimensional, multi-sequence change-point
  problems.
\newblock {\em Statistica Sinica}, 22(4):1507--1538, 2012.

\end{thebibliography}
